\documentclass[11pt]{amsart}

\usepackage{enumerate}
\usepackage{amsmath,amsthm,verbatim,amssymb,amsfonts,amscd, graphicx}
\usepackage{tikz-cd} 
\usepackage{graphics}
\usepackage{mathrsfs}
\usepackage{relsize}
\usepackage{enumitem}
\usepackage{csquotes}
\usepackage{amsthm}
\usepackage{thmtools}
\usepackage[utf8]{inputenc}
\usepackage[sort cites=true, backend=biber]{biblatex}
\usepackage{biblatex}
\tikzset{
	symbol/.style={
		draw=none,
		every to/.append style={
			edge node={node [sloped, allow upside down, auto=false]{$#1$}}}
	}
}

\usepackage{hyperref}

\begin{document}

\title{On the non-uniqueness of maximal purely wild extensions}
\author{Arpan Dutta}
\def\NZQ{\mathbb}               
\def\NN{{\NZQ N}}
\def\QQ{{\NZQ Q}}
\def\ZZ{{\NZQ Z}}
\def\RR{{\NZQ R}}
\def\CC{{\NZQ C}}
\def\AA{{\NZQ A}}
\def\BB{{\NZQ B}}
\def\PP{{\NZQ P}}
\def\FF{{\NZQ F}}
\def\GG{{\NZQ G}}
\def\HH{{\NZQ H}}
\def\UU{{\NZQ U}}
\def\P{\mathcal P}

%
%
\let\union=\cup
\let\sect=\cap
\let\dirsum=\oplus
\let\tensor=\otimes
\let\iso=\cong
\let\Union=\bigcup
\let\Sect=\bigcap
\let\Dirsum=\bigoplus
\let\Tensor=\bigotimes

\theoremstyle{plain}
\newtheorem{Theorem}{Theorem}[section]
\newtheorem{Lemma}[Theorem]{Lemma}
\newtheorem{Corollary}[Theorem]{Corollary}
\newtheorem{Proposition}[Theorem]{Proposition}
\newtheorem{Problem}[Theorem]{}
\newtheorem{Conjecture}[Theorem]{Conjecture}
\newtheorem{Question}[Theorem]{Question}

\theoremstyle{definition}
\newtheorem{Example}[Theorem]{Example}
\newtheorem{Examples}[Theorem]{Examples}
\newtheorem{Definition}[Theorem]{Definition}

\theoremstyle{remark}
\newtheorem{Remark}[Theorem]{Remark}
\newtheorem{Remarks}[Theorem]{Remarks}

\newcommand{\trdeg}{\mbox{\rm trdeg}\,}
\newcommand{\sep}{\mathrm{sep}}
\newcommand{\ac}{\mathrm{ac}}
\newcommand{\ins}{\mathrm{ins}}
\newcommand{\res}{\mathrm{res}}
\newcommand{\nor}{\mathrm{nor}}
\newcommand{\ab}{\mathrm{ab}}
\newcommand{\Gal}{\mathrm{Gal}\,}
\newcommand{\ch}{\mathrm{char}\,}
\newcommand{\Aut}{\mathrm{Aut}\,}
\newcommand{\dist}{\mathrm{dist}\,}

\newcommand{\n}{\par\noindent}
\newcommand{\nn}{\par\vskip2pt\noindent}
\newcommand{\sn}{\par\smallskip\noindent}
\newcommand{\mn}{\par\medskip\noindent}
\newcommand{\bn}{\par\bigskip\noindent}
\newcommand{\pars}{\par\smallskip}
\newcommand{\parm}{\par\medskip}
\newcommand{\parb}{\par\bigskip}

\let\epsilon\varepsilon
\let\phi=\varphi
\let\kappa=\varkappa

\def \a {\alpha}
\def \b {\beta}
\def \s {\sigma}
\def \d {\delta}
\def \g {\gamma}
\def \o {\omega}
\def \l {\lambda}
\def \th {\theta}
\def \D {\Delta}
\def \G {\Gamma}
\def \O {\Omega}
\def \L {\Lambda}
%
%
\textwidth=15cm \textheight=22cm \topmargin=0.5cm
\oddsidemargin=0.5cm \evensidemargin=0.5cm \pagestyle{plain}

\address{Department of Mathematics, IISER Mohali,
	Knowledge City, Sector 81, Manauli PO,
SAS Nagar, Punjab, India, 140306.}
\email{arpan.cmi@gmail.com}

\date{11.14.2020}

\thanks{The author would like to thank Franz-Viktor Kuhlmann for his unwavering support, suggestions and helpful comments. }



\maketitle

\begin{abstract}
	In this paper we illustrate certain criteria which are sufficient for a henselian valued field to admit non-isomorphic maximal purely wild extensions. 
\end{abstract}

\section{Introduction}\label{Section Intro}

In this paper we will consider fields equipped with Krull valuations. A valued field $(K, \nu)$ will denote a field $K$ endowed with the valuation $\nu$. The valuation ring is a local ring $\mathcal{O}_K$ with the unique maximal ideal $\mathcal{M}_K$. We will denote the value group as $\nu K$ and the residue field as $K \nu$. The value of an element $x \in K$ will be denoted by $\nu x$. If $x \in \mathcal{O}_K$, then the residue of $x$ will be denoted by $x \nu$. The residue of a polynomial $f(X) \in \mathfrak{O}_K[X]$ in $K\nu[X]$ is denoted as $f\nu (X)$. We will denote a valued field extension as $(L|K, \nu)$ where $L|K$ is an extension of fields, $\nu$ is a valuation on $L$ and $K$ is endowed with the restriction of $\nu$ to $K$. We say $p$ is the \textbf{characteristic exponent} of $(K,\nu)$ if $p$ = $\ch K\nu$ when $\ch K\nu >0$, and $p=1$ otherwise. For basic facts about valuation theory, refer to [\ref{Abh book}, \ref{Endler book}, \ref{Engler book}, \ref{Kuh book}, \ref{ZS2}].

\pars The algebraic closure of a field $k$ will be denoted as $k^\ac$ and the separable closure as $k^\sep$. The perfect closure of a field $k$ with $\ch k=p>0$ is denoted as $k^{\frac{1}{p^\infty}}$. Given a polynomial $f(X)\in k[X]$, we say that the field $k$ is not $f$-closed in $f(k) \subsetneq k$. Equivalently, there exists $c\in k$ such that the polynomial $f(X)-c$ has no roots in $k$. For a prime $p\in\NN$, a $p$-extension is a (possibly infinite) Galois extension whose Galois group is a $p$-group. For any two fields $L$ and $F$ embedded in a common overfield $\O$, the compositum of $L$ and $F$ is the least subfield of $\O$ containing $L$ and $F$, and it will be denoted by $L.F$. Similarly, the compositum of $\nu L$ and $\nu F$ within $\nu \O$ is the least subgroup of $\nu \O$ containing $\nu L$ and $\nu F$, and it will be denoted by $\nu L + \nu F$. The compositum of two subgroups $A$ and $B$ in a group $G$ endowed with the multiplicative notation will be denoted by $A.B$. If $A$ is normal in $G$ then we have that $A.B=\{ ab \mid a\in A, \, b\in B \}$. The $p$-divisible closure of a group $\G$ is denoted as $\frac{1}{p^\infty} \G$.

\pars A valued field $(K, \nu)$ is said to be \textbf{henselian} if the valuation ring is henselian. Equivalently, $(K, \nu)$ admits a unique extension of $\nu$ to the algebraic closure $K^\ac$ and hence to any algebraic extension of $K$ [\ref{Nagata henselian}]. An algebraic extension of a henselian valued field is again henselian. For any algebraic extesnion $L|K$ of a henselian valued field $(K,\nu)$, we will denote $(L|K,\nu)$ to mean that $L$ is endowed with the unique extension of $\nu$ from $K$ to $L$. For a valued field $(K,\nu)$ admitting an unique extension of $\nu$ to a finite extension $L|K$, we have the \textbf{Lemma of Ostrowski} which gives the relation $[L:K] = p^\d [\nu L:\nu K][L\nu :K\nu] $ for some $\d \in \NN$. The value $p^\d$ is said to be the \textbf{defect} of the extension $(L|K,\nu)$ and will be denoted as $d(L|K,\nu)$. An extension is said to be \textbf{defectless} if $p^\d =1$. Defect satisfies the following multiplicative formula: for finite extensions $(L_1|K,\nu)$ and $(L_2|L_1, \nu)$, we have $d(L_2|K,\nu) = d(L_2|L_1,\nu) d(L_1|K,\nu)$.

\pars An algebraic extension $(L|K,\nu)$ of henselian valued fields is said to be \textbf{tame} if every finite subextension $(E|K,\nu)$ satisfies the following conditions: 
\sn (T1) $\ch K\nu$ does not divide $[\nu E : \nu K]$,
\n (T2) the residue field extension $E\nu |K\nu$ is separable. 
\n (T3) $(E|K,\nu)$ is defectless, that is, $[E:K] = [\nu E:\nu K][E\nu :K\nu]$. 
\newline $(L|K,\nu)$ is said to be \textbf{tamely ramified} if only (T1) and $(T2)$ holds for any finite subextension. The extension $(L|K,\nu)$ is \textbf{unramified} if it is tamely ramified and $\nu L=\nu K$. 

\pars An arbitrary extension of valued fields $(F|K,\nu)$ is said to be \textbf{immediate} if $\nu F = \nu K$ and $F \nu = K \nu$. A valued field $(M,\nu)$ is said to be \textbf{maximal} if it does not admit any nontrivial immediate extension. $(M,\nu)$ is said to be \textbf{algebraically maximal} if it does not admit any nontrivial immediate algebraic extension. It was shown by Krull that every valued field admits a maximal immediate extension [cf. \ref{Krull}]. The problem of the uniqueness of maximal immediate extensions was considered by Kaplansky in [\ref{Kaplansky}]. He proved that a sufficient condition for uniqueness is given by the so called \enquote{hypethesis A}. It has been shown [cf. \ref{Delon}, \ref{Kuh additive polynomials}, \ref{Whaples}] that hypothesis A consists of the following three conditions:
\sn (A1) $K\nu$ does not admit finite separable extensions whose degrees are divisible by $p$,
\n (A2) $\nu K = \frac{1}{p^\infty} \nu K$,
\n (A3) $K\nu$ is perfect.
\newline It is known that \enquote{hypothesis A} is not necessary for the uniqueness of maximal immediate extensions. A class of non-Kaplansky fields with unique maximal immediate extensions is given in Theorem 1.2 of [\ref{Anna Ku on maximal immediate extns of valued fields}].  

\pars Let $(K,\nu)$ be a henselian valued field with $\ch Kv =p >0$. An algebraic extension $(L|K,\nu)$ is said to be \textbf{purely wild} if $\nu L/\nu K$ is a $p$-group and $L\nu|K\nu$ is a purely inseparable extension. An equivalent definition is that the algebraic extension $L$ is linearly disjoint to the absolute ramification field $K^r$ [cf. Section \ref{Section valn and field theoretic facts}] over $K$, which is same as the condition $L\sect K^r = K$. In particular, purely inseparable extensions of valued fields and immediate extensions are purely wild extensions. The following result, shown by Matthias Pank in [\ref{Kuh-Pank-Roquette}], proves the existence of maximal purely wild extensions:  

\begin{Theorem}
	Let $(K,\nu)$ be a henselian valued field with $\ch K\nu =p>0$. Then there exists a valued field extension $(L|K,\nu)$ such that $(L|K,\nu)$ is purely wild and $L.K^r =K^\ac$. The degree of any finite subextension of $L|K$ is a power of $p$. Further, $\nu L = \frac{1}{p^\infty} \nu K$ and $L\nu = K\nu^{\frac{1}{p^\infty}}$.
\end{Theorem}
Any purely wild extension can thus be embedded in a maximal one. In Theorem 5.1 of [\ref{Kuh-Pank-Roquette}] it has been shown that the maximal purely wild extensions are unique up to valuation preserving isomorphisms if the residue field $K\nu$ satisfies (A1). It is natural to inquire if the condition is also necessary to achieve uniqueness. A preliminary answer to this question is obtained in Proposition 8.2 of [\ref{Kuh-Pank-Roquette}].

\begin{Proposition}
	Let $(K,\nu)$ be a henselian valued field with $\ch K\nu=p>0$ and suppose that $K\nu$ admits a finite extension whose degree is divisible by $p$. Also suppose that there is a nontrivial separable algebraic extension $(S|K,\nu)$ which is purely wild. Then there exists a finite tame extension $(E|K,\nu)$ with $\gcd ([E:K],\, p) =1$ such that $E$ admits at least two maximal purely wild algebraic extensions which are not $E$-isomorphic. 
\end{Proposition}
In this paper we investigate the question under which precise conditions we can take $E=K$. 
The following result is implicit in the works of Kuhlmann, Pank and Roquette [\ref{Kuh-Pank-Roquette}], though it has not been explicitly mentioned:

\begin{Proposition}\label{Proposition Galois purely wild}
	 Let $(K(\th)|K,\nu)$ be a purely wild extension of henselian valued fields with $\ch K\nu=p>0$. Suppose that $K(\th)|K$ is Galois and that $K\nu$ is not Artin-Schreier closed. Then $(K,\nu)$ does not admit a unique maximal purely wild extension.
\end{Proposition}
In Section \ref{Section constructive non-uniqueness} we prove far reaching generalizations of the aforementioned observation. For a field $K$ with $\ch K=p>0$, a Galois extension of degree $p$ is an Artin-Schreier extension [Theorem 6.4, \S 6, Chapter VI, \ref{Lang}]. If $\ch K=0$, $\ch K\nu =p$ and $K$ contains one, hence all, primitive $p$-th roots of unity, then a Galois degree $p$ extension is a Kummer extension [Theorem 6.2, \S 6, Chapter VI, \ref{Lang}]. Artin-Schreier polynomials are special cases of \textbf{$p$-polynomials}, that is, polynomials of the form $\sum_{i=0}^{n} a_i X^{p^i} -c \in K[X]$. When $\ch K=p>0$. there is a tight connection between $p$-polynomials and minimal purely wild extensions, that is, purely wild extensions which do not admit proper nontrivial subextensions. It is illustrated in the following result which is due to Florian Pop [Theorem 13, \ref{Kuh additive polynomials}]:
\begin{Theorem}
	Let $(L|K,\nu)$ be a minimal purely wild extension of henselian valued fields with $\ch K=p>0$. Then there is a generator $\th$ of $L|K$ such that the minimal polynomial of $\th$ over $K$ is of the form $f(X) = \sum_{i=0}^{n} a_i X^{p^i} -c$. We can further assume that $a_i \in \mathfrak{O}_K$ for each $i$.
\end{Theorem} 
In Section \ref{Section constructive non-uniqueness} we obtain the following result:
\begin{Theorem}\label{Thm introduction}
	Let $(K(\th)|K,\nu)$ be a minimal purely wild extension of henselian valued fields with $\ch K=p>0$. We can consider the minimal polynomial of $\th$ over $K$ to be of the form $f(X) = \mathfrak{A}(X) -c$, where
	\[ \mathfrak{A}(X) := \sum_{i=0}^{n} a_i X^{p^i} \in \mathfrak{O}_K[X]. \]
	Suppose that $\nu a_0=\nu a_n=0$ and that $K\nu$ is not $\mathfrak{A}\nu$-closed. Then $(K,\nu)$ does not admit a unique maximal purely wild extension.  
\end{Theorem}
Observe that the polynomial $\mathfrak{A}(X) := \sum_{i=0}^{n} a_i X^{p^i}$ satisfies the condition $\mathfrak{A}(x+y) = \mathfrak{A}(x)+ \mathfrak{A}(y)$ for all $x,y$ and hence $\mathfrak{A}(X)$ is an \textbf{additive} polynomial. The Frobenius endomorphism fails to hold over $K$ when $\ch K=0$ and hence the polynomial $\mathfrak{A}(X)$ is not additive in this case. However we observe that for sufficiently large values of the generator $\th$ we can obtain analogous results in the mixed characteristic case: 
\begin{Theorem}
	Let $(K(\th)|K,\nu)$ be a purely wild extension of henselian valued fields with $\ch K=0$ and $\ch K\nu=p>0$. Suppose that the minimal polynomial of $\th$ over $K$ is of the form $f(X):=\mathfrak{A}(X)-c$, where,
	\[ \mathfrak{A}(X):= \sum_{i=0}^{n} a_i X^{p^i}\in\mathfrak{O}_K[X]. \]
	Suppose that $\nu a_0=\nu a_n=0$, $\nu\th> - \frac{\nu p}{p^n}$ and that $K\nu$ is not $\mathfrak{A}\nu$-closed. Then $(K,\nu)$ does not admit a unique maximal purely wild extension. If $(K(\th)|K,\nu)$ is an immediate extension, then $(K,\nu)$ does not admit a unique maximal immediate algebraic extension.
\end{Theorem}
These results are obtained by constructing a purely wild extension $(K(\a)|K,\nu)$ which does not lie in the maximal purely wild extension containing $K(\th)$. We start with the given minimal polynomial of $\th$ and use the conditions on the residue field to find a suitable minimal polynomial for $\a$. This same approach is also used to illustrate that if $(K(\th)|K,\nu)$ is a Kummer purely wild extension and if $K\nu$ is not Artin-Schreier closed, then $(K,\nu)$ does not admit a unique maximal purely wild extension. The ramification theoretic classification of extensions generated by roots of $p$-polynomials is central to this construction. Our classification of Artin-Schreier and Kummer extensions is in essence similar to the one presented in [\ref{Kuh book}]. A principal object in the classification is the set $\nu(\th-K) := \{ \nu(\th-c)\mid c\in K \}$. This set is tightly connected to the notion of \enquote{distance} which plays an important role in the classification of Artin-Schreier defect extensions [\ref{Kuh A-S extensions and defectless fields paper}] and Kummer defect extensions [\ref{Kuh gdr}].

\pars In the situation of Theorem \ref{Thm introduction}, we observe that $(E|K,\nu)$ is a Galois tame extension where $E$ is the splitting field of $\mathfrak{A}(X)$ over $K$. Further let $h(X):= \mathfrak{A}(X)-t\in K[X]$ where $\mathfrak{A}\nu(X)-t\nu$ does not have any roots in $K\nu$. Setting $F:= K(\b)$ where $h(\b)=0$, we observe that $(F|K,\nu)$ is a tame extension and $F.E|K$ is Galois. Setting $L:= K(\th)$ we again note that $L.E|K$ is a Galois extension. The following result, proved in Section \ref{Section group theoretic non-unique} shows that under certain additional criteria, we can achieve non-uniqueness irrespective of the characteristic of the base field. 
\begin{Theorem}
Let $(K,\nu)$ be a henselian valued field with $\ch K\nu=p>0$. Let $(L|K,\nu)$ be a separable nontrivial purely wild extension and $(F|K,\nu)$ be a nontrivial tame extension. Suppose that $(E|K,\nu)$ is an abelian Galois tame extension such that $L.E|K$ and $F.E|K$ are Galois extensions. Further suppose that the following conditions are satisfied:
\sn (1) $\th: \Gal(F.E|K) \longrightarrow \Gal(L.E|K)$ is a continuous isomorphism,
\n (2) $\th_{\res} : \Gal(F.E|E) \longrightarrow \Gal(L.E|E)$ is a continuous isomorphism. 
\newline Let $(L.E|K)^\ab$ denote the maximal abelian Galois subextension of $L.E|K$. Then the following cases are possible:
\sn (a) $E\subsetneq (L.E|K)^\ab$. Then $(K,\nu)$ does not admit a unique maximal purely wild extension. 
\n (b) $E=(L.E|K)^\ab$. Then $(E,\nu)$ does not admit a unique maximal purely wild extension. 
\end{Theorem}

In the statement of the next theorem, $G_r$ denotes the absolute ramification group [cf. Section \ref{Section valn and field theoretic facts}] and $[G_r,H]$ denotes the subgroup of $\Gal(K^\sep|K)$ generated by all commutators $ghg^{-1}h^{-1}$ where $g\in G_r$ and $h\in H$. As a consequence we rediscover Proposition \ref{Proposition Galois purely wild}.

\begin{Theorem}
	Let $(L|K,\nu)$ be a nontrivial separable purely wild extension of henselian valued fields with $\ch K\nu=p>0$. Set $H:= \Gal(K^\sep|L)$ and $F:= (K^\sep)^{[G_r,H]}$. Suppose that $L$ and $F$ are not linearly disjoint over $K$. Then $(K,\nu)$ does not admit a unique maximal purely wild extension if the residue field $K\nu$ is not Artin-Schreier closed. 
\end{Theorem}

\par We finally mention that the non-uniqueness of maximal immediate extensions has been studied extensively in [\ref{Warner non-uniqueness max imm extns}]. An extreme case of non-uniqueness of maximal immediate extensions is constructed in Theorem 1.6 of [\ref{Anna Ku alg independence of elements in completions and max imm extns}].  


\section{Valuation and field theoretic facts} \label{Section valn and field theoretic facts}
In this section, we recall some results from field theory [cf. \ref{Lang}] and valuation theory [cf. \ref{Abh book}, \ref{Endler book}, \ref{Engler book}, \ref{Kuh book}, \ref{ZS2}]. Two field extensions $K_1|k$ and $K_2|k$ which are embedded in some overfield of $k$ are said to be \textbf{linearly disjoint} if any finite set of elements of $K_1$ which are linearly independent over $k$ are still such over $K_2$. When both extensions are finite, an equivalent characterization is $[K_1.K_2 :k] = [K_1:k][K_2:k]$. If $K_1|k$ is Galois and $K_2|k$ is an arbitrary extension, then a necessary and sufficient condition for linear disjointness is $K_1 \sect K_2 = k$. 
\newline If $K_1|k$ is separable algebraic and $K_2|k$ is purely inseparable, then $K_1$ and $K_2$ are linearly disjoint over $k$. Further if $F|k$ is any arbitrary extension, then $K_1.F|F$ is separable algebraic [Theorem 4.5, Chapter V, \ref{Lang}] and $K_2.F|F$ is purely inseparable [Proposition 6.5, Chapter V, \ref{Lang}].

\pars Consider the normal extension $(K^\ac|K,\nu)$ of valued fields for an arbitrary field $K$ and let $G := \Gal(K^\sep|K)$. Equivalently we will identify $G$ with the group of all automorphisms of $K^\ac$ over $K$. The absolute decomposition group of the extension is defined as,
\[  G_d := \{\s \in G \mid \nu \circ \s = \nu \text{ on } K^\sep \}.\]
The absolute inertia group is defined as,
\[G_i := \{\s \in G \mid \nu(\s x - x) > 0 \, \forall \, x \in \mathcal{O}_{K^\sep} \}.\]
The absolute ramification group is defined as, 
\[ G_r := \{\s \in G \mid \nu(\s x - x) > \nu x \, \forall \, x \in K^\sep \setminus \{0\} \}. \]
The corresponding fixed fields in $K^\sep$ will be denoted as $K^d$, $K^i$ and $K^r$ and they are called the \textbf{absolute decomposition field}, \textbf{absolute inertia field} and the \textbf{absolute ramification field} of $(K,\nu)$ respectively. We have that $G_r \unlhd G_i \unlhd G_d$ and $G_r \unlhd G_d$. Further these are closed normal subgroups. We also have the continuous isomorpshism $G_d/G_i = \Gal(K^i|K^d) \iso \Gal((K\nu)^\sep|K\nu)$. For an arbitrary algebraic extension of valued fields $(L|K,\nu)$, we have that $L^d = L.K^d$, $L^i = L.K^i$ and $L^r = L.K^r$. The \textbf{henselization} of a valued field $(K,\nu)$ is the smallest henselian field containing $K$, and we can consider it to be the same as the absolute decomposition field $K^d$. So when $(K,\nu)$ is henselian, $K^r|K$ and $K^i|K$ are Galois extensions. We have the following important result [Theorem 11.1, \ref{Kuh vln model}]:
\begin{Theorem}\label{Theorem maximal tame}
	If $(K,\nu)$ is henselian, then $(K^r,\nu)$ is the unique maximal tame extension of $(K,\nu)$ and $(K^i,\nu)$ is the unique maximal defectless and unramified extension of $(K,\nu)$.
\end{Theorem}
We note this following equivalent condition of Hensel's Lemma, often referred to as the strong Hensel's Lemma. 

\begin{Lemma}
	Let $(K,\nu)$ be a henselian valued field. For any polynomial $f(X) \in \mathfrak{O}_K[X]$ the following holds true: if there is a factorization $f\nu(X) = g\nu(X) h\nu(X)$ such that $g\nu(X)$ is relatively prime to $h\nu(X)$ in $K\nu[X]$, then there exists $g(X), \, h(X) \in \mathfrak{O}_K[X]$ with residues $g\nu(X)$ and $h\nu(X)$ such that $f(X)=g(X)h(X)$ and $\deg g(X) = \deg g\nu(X)$.
\end{Lemma}

\begin{Remark}\label{Remark hensels lemma}
	Let $(K,\nu)$ be a henselian valued field with $f(X) \in \mathfrak{O}_K[X]$ such that $\deg f=\deg f\nu$. Observe that $f\nu (X)\in K\nu[X]$ which is an unique factorization domain and hence we have an irreducible decomposition $f\nu = (f_1 \nu)^{e_1} \dotsc (f_r\nu)^{e_r}$. By repeated application of the strong Hensel's Lemma, we have a decomposition $f=f_1^{e_1}\dotsc f_r^{e_r}$ where $f_i(X) \in \mathfrak{O}_K[X]$ and $\deg f_i =\deg f_i \nu$ for each $i$.
\end{Remark}

In the case of valued fields, we extend the notion of being linearly disjoint to that of being valuation disjoint. An extension $(L|K,\nu)$ is said to be \textbf{valuation disjoint} from $(F|K,\nu)$ if it satisfies the following two conditions:
\sn (1) $\nu L \sect \nu F = \nu K$,
\n (2) $L\nu$ and $F\nu$ are linearly disjoint over $K\nu$. 
\newline It follows directly that tame extensions are valuation disjoint from purely wild extensions. We have the following general result [cf. Chapter 11.9, \ref{Kuh book}].
\begin{Theorem}
	Let $(\O|K,\nu)$ be a valued field extension. Suppose that $(L|K,\nu)$ is a defectless algebraic subextension which is valuation disjoint from an arbitrary subextension $(F|K,\nu)$. Then, 
	\sn (1) $L|K$ and $F|K$ are linearly disjoint,
	\n (2) there is an unique extension of $\nu$ from $F$ to $L.F$ and $d(L.F|F,\nu) = 1$,
	\n (3) $\nu(L.F) = \nu L + \nu F$ and $(L.F) \nu = L\nu .F\nu$.
\end{Theorem}

As a consequence we obtain the following result which we will require in the later sections. 
\begin{Lemma}\label{Lemma valn disj tame and purely wild}
	Let $(K,\nu)$ be a henselian valued field. Let $(L|K,\nu)$ be a tame extension and $(F|K,\nu)$ be purely wild. Then $L$ and $F$ are linearly disjoint over $K$, $(L.F|F,\nu)$ is tame and $(L.F|L,\nu)$ is purely wild. In addition, if $(L|K,\nu)$ is unramified, then $(L.F|F,\nu)$ is also unramified.
\end{Lemma}

\begin{proof}
	The extensions $(L|K,\nu)$ and $(F|K,\nu)$ satisfy the assumptions of the above theorem, hence it follows directly that $L|K$ and $F|K$ are linearly disjoint, $d(L.F|F,\nu) = 1$ and $(L.F)\nu|F\nu = L \nu . F \nu |F\nu$ is a separable extension. Further, $\nu(L.F)/\nu F = (\nu L+\nu F)/\nu F \iso \nu L/\nu K$ and hence $(L.F|F,\nu)$ is tame. Similarly, $(L.F|L,\nu)$ is purely wild. Finally, if $(L|K,\nu)$ is defectless and unramified, then $(L.F|F,\nu)$ is also defectless and unramified.  
\end{proof}

 \pars An \textbf{Artin-Schreier} polynomial $f(X)$ over a field $K$ with $\ch K =p>0$ is a polynomial of the form $f(X) = X^p-X-a$. Given a root $\th$ of $f$, one obtains all the roots of $f$, namely $\{\th+i \mid i \in \FF_p \}$. So $f(X)$ is separable over $K$, and it either splits or is irreducible over $K$. If $\th \notin K$, then $K(\th)|K$ is a Galois extension and we will say that $\th$ is an Artin-Schreier generator of $K(\th)|K$. The field $K$ is said to be Artin-Schreier closed if it does not admit any irreducible Artin-Schreier polynomials.
\pars Let $K$ be a field containing a primitive $p$-th root of unity with $\ch K=0$. A \textbf{Kummer} polynomial is a polynomial of the form $f(X):= X^p-a\in K[X]$. Thus $f(X)$ either splits or is irreducible over $K$. If $f(\th)=0$ and $\th\notin K$, then $K(\th)|K$ is a Galois extension and we say that $\th$ is a Kummer generator.

\pars For a valued field extension $(K(\th)|K,\nu)$, we define $\nu(\th-K) := \{ \nu(\th-c) \mid c \in K \}$ and $\L^L(\th,K) := \nu(\th-K) \sect \nu K$. It can be directly checked that $\L^L(\th,K)$ is an initial segment of $\nu K$. For $\g \in \nu K(\th)$ and a set $S \subseteq \nu K$, we say $\g \geq S$ if $\g \geq x$ for all $x \in K$.

 \begin{Lemma} \label{lemma distance from K}
	Let $(K(\th)|K, \nu)$ be an extension of valued fields. Then either $\nu(\th-K) = \L^L (\th,K)$ or $ \nu(\th-K) = \L^L (\th,K) \union \{\g\}$ where $\g > \L^L (\th,K)$.
\end{Lemma} 

\begin{proof}
	Suppose that $\nu(\th-K) \neq \L^L (\th,K)$. Let $\g = \nu(\th-t) \in \nu(\th-K) \setminus \L^L (\th,K)$ where $t \in K$. Suppose that $\g < \nu(\th-c)$ for some $c \in K$. Then we have that $\g = \nu(t-c) \in \nu K$ which is a contradiction. So $\g > \nu(\th-K) \setminus \{\g\}$. Again let $\g_1 = \nu(\th-t_1) \in \nu(\th-K) \setminus \L^L (\th,K)$ such that $\g_1 \neq \g$. Then $\text{min} \{\g, \g_1\} = \g_1 = \nu(t-t_1) \in \nu K$ which is also a contradiction. So $\nu(\th-K) = \L^L (\th,K) \union \{\g\}$ with $\g > \L^L (\th,K)$. 
\end{proof}

If $(K(\th)|K,\nu)$ is immediate, then $\nu(\th-K) = \L^L(\th,K)$ does not have a maximal element [Lemma 1.1, \ref{Kuh defect}]. The converse is not true in general, but does hold under certain special conditions. The following is Lemma 2.21 of [\ref{Kuh A-S extensions and defectless fields paper}]:
\begin{Lemma}
	Let $(K,\nu)$ be a valued field with $\ch K\nu = p>0$. Suppose that $\nu$ admits an unique extension to $K(\th)$ where $[K(\th):K]=p$. Then the fact that $\nu(\th-K)$ has no maximal element implies that $(K(\th)|K,\nu)$ is immediate.
\end{Lemma}

\pars We will require the following result which is Lemma 14 of [\ref{Kuh generalized stability}].
\begin{Lemma}\label{lemma C^{p-1} = p}
	Let $(K, \nu)$ be a henselian valued field with $\ch K = 0$ and $\ch K \nu = p > 0$. Then there exists $C \in K$ satisfying $C^{p-1} = -p$ if and only if $K$ contains a primitive $p$-th root of unity. 
\end{Lemma}

The proof of this next lemma can be obtained in [\ref{Kuh book}].
\begin{Lemma} \label{lemma 1+a = b^p}
	Let $(K, \nu)$ be a henselian field containing all $p$-th roots of unity. Then for any $a \in K$, $\nu a > \frac{p}{p-1} \nu p$ implies that $1+a = b^p$ for some $b \in K \setminus \{0\}$. 
\end{Lemma}


\section{Constructive conditions for non-uniqueness}\label{Section constructive non-uniqueness}

When the valued field $(K,\nu)$ is of positive characteristic, we had observed in Section \ref{Section Intro} that a minimal purely wild extension $(L|K,\nu)$ is generated by a root of a $p$-polynomial. In this case we can give certain constructive conditions for achieving non-uniqueness. 

We first take a closer look at arbitrary extensions generated by roots of $p$-polynomials.

\begin{Lemma}\label{Lemma extns by p-polynomials}
	Let $(K,\nu)$ be a henselian valued field with $\ch K=p>0$. Let $\mathfrak{A}(X) = \sum_{i=0}^{n} a_i X^{p^i} \in \mathfrak{O}_K[X]$. Suppose that $f(X):=\mathfrak{A}(X)-c \in K[X]$ has no roots in $K$. Let $f(\th)=0$. Further suppose that $\nu a_0=\nu a_n=0$. Then the following statements hold true:
	\sn (1) $\nu(\th-K)\leq 0$.
	\n (2) If $\nu(\th-b)=0$ for some $b\in K$ then $(K(\th)|K,\nu)$ is a defectless unramified extension with $K(\th)\nu =K\nu((\th-b)\nu)$.
	\n (3) If $\nu c=0$ then $\nu\th=0$.
\end{Lemma}

\begin{proof}
	Observe that $\mathfrak{A}(X)$ is an additive polynomial and $\mathfrak{A}(\th) =c$. For any $b\in K$, we define
	\[ f_b(X) := f(X+b) = \mathfrak{A}(X+b)-c = \mathfrak{A}(X)-(c-\mathfrak{A}(b)). \]
	Suppose that $\nu(\th-b)>0$ for some $b\in K$. Then $\nu \mathfrak{A}(\th-b) >0$. Again, $\mathfrak{A}(\th-b) = c-\mathfrak{A}(b)$. So in this case, $f_{b}(X) \in \mathfrak{O}_K[X]$ with $f_b \nu(X) = \mathfrak{A}\nu (X)$. Now the fact that $a_0 \nu \neq 0$ implies that $f_b\nu(X)$ is separable. Hence $0$ is a simple root of $f_b\nu$ in $K\nu$. Since $(K,\nu)$ is henselian, it can be lifted to a simple root of $f_b(X)$ over $K$. But this is not possible since $f(X) = f_b(X-b)$ admits no roots in $K$. So $\nu(\th-b) \leq 0$ for each $b\in K$ and hence $\nu(\th-K) \leq 0$.
	
	\pars Now suppose that $\nu(\th-b) =0$ for some $b\in K$. With the same arguments as above, we obtain that $\nu (c-\mathfrak{A}(b)) =0$. Then $f_b\nu(X)=\mathfrak{A}\nu(X)-(c-\mathfrak{A}(b))\nu$. The fact that $a_n\nu\neq 0$ implies that $\deg f_b\nu=\deg f_b=\deg f$. From Remark \ref{Remark hensels lemma} we obtain a decomposition $f_b = f_1\dotsc f_r$ such that $f_b\nu =f_1\nu \dotsc f_r\nu$ is an irreducible decomposition of $f_b\nu$ and $\deg f_i=\deg f_i\nu$ for each $i$. Now, 
	\[ f(\th)=0 \Longrightarrow f_b(\th-b)=0 \Longrightarrow f_b\nu ((\th-b)\nu)=0. \]
	Without any loss of generality we can assume that $f_1 (\th-b)=0$. Then $f_1\nu ((\th-b)\nu)=0$ and from the irreducibility criterion it follows that $f_1\nu$ is the minimal polynomial of $(\th-b)\nu$ over $K\nu$. Then,
	\[ [K(\th):K] \leq \deg f_1=\deg f_1\nu = [K\nu((\th-b)\nu):K\nu] \leq [K(\th)\nu:K\nu] \leq [K(\th):K]. \]
	Hence all the inequalities in the above expression are equalities. From the lemma of Ostrowski it now follows that,
	\[ [K(\th):K] = [K(\th)\nu:K\nu] = [K\nu ((\th-b)\nu):K\nu] \text{ and } \nu K(\th) = \nu K. \]
	We have observed that $f_b\nu (X)$ is separable over $K\nu$ which further implies the separability of $f_1\nu$ over $K\nu$. Consequently the residue field extension $K\nu((\th-b)\nu)|K\nu$ is separable. Hence $(K(\th)|K,\nu)$ is a defectless and unramified extension with $K(\th)\nu = K\nu ((\th-b)\nu)$.
	\pars The fact that $\nu(\th-K)\leq 0$ implies that $\nu\th \leq 0$. Suppose that $\nu\th<0$. Then for any $i\neq n$, we have that,
	\[ \nu a_i \th^{p^i} =\nu a_i + p^i\nu\th \geq p^i\nu\th >p^n\nu\th. \]
	Hence $\nu\mathfrak{A}(\th)=p^n\nu\th$. Thus $\nu \mathfrak{A}(\th) =\nu c=0$ implies that $\nu\th =0$.
\end{proof}

Observe that Artin-Schreier polynomials satisfy the assumptions of Lemma \ref{Lemma extns by p-polynomials}. We can give a more detailed classification of Artin-Schreier extensions. 

\begin{Lemma}
	Let $(K(\th)|K, \nu)$ be an Artin-Schreier extension of henselian valued fields with $\ch K=p>0$ and $\th^p-\th=a\in K$. Then the following cases are possible. 
	\sn (1) $\nu(\th-K)$ has no maximal element. In this case, $(K(\th)|K, \nu)$ is an immediate Artin-Schreier extension. 
	\n (2) $\nu(\th-K)$ has $0$ as the maximal element. In this case, $K(\th)\nu | K\nu$ is an Artin-Schreier extension and $(K(\th)|K, \nu)$ is a defectless unramified extension.
	\n (3) $\nu(\th-K)$ has a maximal element $\g < 0$. In this case, $(K(\th)|K, \nu)$ is a defectless purely wild extension. 
\end{Lemma}

\begin{proof}
	$(1)$ is a direct application of [Lemma 2.21, \ref{Kuh A-S extensions and defectless fields paper}].
	\pars $(2)$ follows from Lemma \ref{Lemma extns by p-polynomials}.
	\pars Now suppose that $\g = \nu(\th-t)$ is the maximal value of $\nu(\th-K)$ with $\g <0$. Then $\nu(a-t^p+t) = \nu((\th-t)^p-(\th-t))= p \g \in \nu K$. If $\g \notin \nu K$ then we have $[\nu K(\th) : \nu K] = p$ and thus $\nu K(\th)/\nu K$ is a $p$-group. Hence $(K(\th)|K, \nu)$ is defectless and purely wild. 
	\newline Now suppose that $\g = \nu(\th-t)\in \nu K$ with $\g = \nu c$ for some $c \in K$. Let $\eta = \frac{\th-t}{c}$. Thus $\nu (\eta) = 0$ and $\eta \nu \in K(\th)\nu $. We have,
	\[ \nu(\eta^p - \frac{a-t^p+t}{c^p}) = \nu (\frac{\th-t}{c^p}) = (1-p) \g > 0  \]
	and hence $\eta^p \nu = (\frac{a-t^p+t}{c^p}) \nu \in K \nu$. If $\eta \nu \in K \nu$, then $\nu((\frac{\th -t}{c}) -b) > 0$ for some $b \in K$. But this implies that $\nu(\th-t-bc) > \nu c = \g$ which contradicts the maximality of $\g$. So $\eta \nu \notin K \nu$. Since $\eta^p \nu \in K \nu$, we thus have that $K(\th)\nu = K \nu (\eta \nu)$ and $K(\th)\nu | K \nu$ is a purely inseparable extension of degree $p$. So $(K(\th)|K, \nu)$ is a defectless purely wild extension.
\end{proof}

\begin{Corollary}
	Let $(K(\th)|K,\nu)$ be an Artin-Schreier extension of henselian discretely valued fields of rank $1$ with $\ch K=p>0$. Then $(K(\th)|K, \nu)$ is defectless.
\end{Corollary}
Indeed a more general result actually holds. It is known that any finite separable extension over a discretely valued field of rank $1$ is a defectless extension. The conclusions of the above corollary fail to hold for non-discrete valuations, even in the rank one case. We provide an example illustrating that below. 

\begin{Example}
	Let $(K, \nu)$ be the power series field $\FF_{p} ((t))$ equipped with the $t$-adic valuation. Then $\nu K = \ZZ$ and $K \nu = \FF_p$. Let $(L, \nu)$ be the Puiseux series field equipped with the canonical valuation. Then $(L, \nu)$ is henselian, $\nu L = \QQ$ and $L \nu = \FF_p$. Let $\th \in L^{\ac}$ be such that $\th^p - \th = \frac{1}{t}$. Then $(L(\th)|L, \nu)$ is an immediate extension [cf. Example 15, \ref{Kuh vln model}]. Note that $(L, \nu)$ is a non-discrete valuation of rank $1$. 
\end{Example}
The conclusions of the corollary also fail to hold for discrete valuations of higher rank. An example is furnished at Example 3.18 of [\ref{Kuh defect}].

\pars We now show that if a henselian valued field of positive characteristic admits a minimal purely wild extension satisfying the assumptions of Lemma \ref{Lemma extns by p-polynomials}, then under an additional hypothesis on the residue field, we achieve non-uniqueness of maximal purely wild extensions. 

\begin{Theorem}\label{Theorem non-unique eq char}
	Let $(K(\th)|K,\nu)$ be a minimal purely wild extension of henselian valued fields with $\ch K=p>0$. We can consider the minimal polynomial of $\th$ over $K$ to be of the form $f(X) = \mathfrak{A}(X) -c$, where
	\[ \mathfrak{A}(X) := \sum_{i=0}^{n} a_i X^{p^i} \in \mathfrak{O}_K[X]. \]
	Suppose that $\nu a_0=\nu a_n=0$ and that $K\nu$ is not $\mathfrak{A}\nu$-closed. Then $(K,\nu)$ does not admit a unique maximal purely wild extension. 
\end{Theorem}

\begin{proof}
	Let the polynomial $\mathfrak{A}\nu(X)-t\nu$ does not admit any roots in $K\nu$ for some $t\in K$. Thus $\nu t=0$. Let us define $g(X):= \mathfrak{A}(X)-(c+t)$ and $h(X):= \mathfrak{A}(X)-t$. Let $\a\in K^\ac$ such that $g(\a) =0$. Then,
	\[ \mathfrak{A}(\a-\th) = \mathfrak{A}(\a)-\mathfrak{A}(\th) = t \Longrightarrow h(\a-\th)=0. \]
	Suppose that $h(X)$ admits a root $\b\in K$. Then the fact that $\nu t=\nu a_n=0$ implies that $\nu\b=0$, which contradicts that $h\nu(X)$ does not admit any roots in $K\nu$. Hence $h(X)$ does not admit any roots in $K$. From Lemma \ref{Lemma extns by p-polynomials} we then obtain that $\nu(\a-\th)=0$ and that $(K(\a-\th)|K,\nu)$ is a defectless unramified extension with $K(\a-\th)\nu=K\nu((\a-\th)\nu)$. Thus $(K(\th)|K,\nu)$ and $(K(\a-\th)|K,\nu)$ are valuation disjoint extensions. From Lemma \ref{Lemma valn disj tame and purely wild} we now obtain that $(K(\a,\th)|K(\th),\nu)$ is a defectless and unramified extension. Further,
	\[ [K(\a,\th):K(\th)]=[K(\a-\th):K]=[K\nu((\a-\th)\nu):K\nu]. \]
	\pars Now suppose that $f(X)$ admits a root $\s\th$ in $K(\a)$ for some $\s\in\Aut(K^\ac|K)$. The fact that $(K,\nu)$ is a henselian valued field implies that any $K$-automorphism of $K^\ac$ preserves the valuation $\nu$. So $(K(\s\th)|K,\nu)$ is also a minimal purely wild extension with the same minimal polynomial $f(X)$ over $K$. So without any loss of generality we can assume that $\th\in K(\a)$. Then $K(\th)\subseteq K(\a)$. Hence,
	\[ [K(\a):K]\leq\deg g=\deg f=[K(\th):K]\Longrightarrow K(\a)=K(\th). \] 
	But this implies that $K(\a-\th)=K$ which contradicts the fact that $h(X)$ does not admit any roots in $K$. Hence $f(X)$ does not admit any roots in $K(\a)$. From Lemma \ref{Lemma extns by p-polynomials} we now obtain that $(K(\a,\th)|K(\a),\nu)$ is a defectless unramified extension with $K(\a,\th)\nu=K(\a)\nu((\a-\th)\nu)$.
	\pars Now,
	\[ [K(\a,\th):K]=[K(\a,\th):K(\th)][K(\th):K]=[K(\a,\th):K(\a)][K(\a):K]. \]
	We have that $[K(\a):K]\leq [K(\th):K]$. Again, 
	\[  [K(\a,\th):K(\a)] = [K(\a)\nu((\a-\th)\nu):K(\a)\nu] \leq [K\nu((\a-\th)\nu):K\nu]= [K(\a,\th):K(\th)]. \]
	Thus we have the following equalities, 
	\[ [K(\th):K]=[K(\a):K] ,\, \, [K(\a,\th):K(\th)]=[K(\a,\th):K(\a)]. \]
	\pars We consider the separable degree of the extension $K(\a,\th)\nu|K\nu$, denoted by $[K(\a,\th)\nu:K\nu]_{\sep}$. From the multiplicative property of the separable degree [Theorem 4.1, Chapter V, \S 4, \ref{Lang}], we obtain that,
	\[ [K(\a,\th)\nu:K\nu]_{\sep} = [K(\a,\th)\nu:K(\th)\nu]_{\sep} [K(\th)\nu:K\nu]_{\sep} = [K(\a,\th)\nu:K(\a)\nu]_{\sep}[K(\a)\nu:K\nu]_{\sep}. \]
	Now, $K(\th)\nu|K\nu$ is purely inseparable, hence has separable degree one. Further, $(K(\a,\th)|K(\th),\nu)$ and $(K(\a,\th)|K(\a),\nu)$ are defectless and unramified extensions. The above expression now reads as,  
	\[ [K(\a,\th)\nu:K\nu]_{\sep} = [K(\a,\th):K(\th)] = [K(\a,\th):K(\a)][K(\a)\nu:K\nu]_{\sep}. \]
	Thus $[K(\a)\nu:K\nu]_{\sep} = 1$, that is, $K(\a)\nu|K\nu$ is purely inseparable. Again, $\nu K(\th)=\nu K(\th,\a)=\nu K(\a)$. So, $\nu K(\a)/\nu K$ is a $p$-group and hence $(K(\a)|K,\nu)$ is a purely wild extension.
	\pars The purely wild extension $(K(\th)|K,\nu)$ is embedded in some maximal purely wild extension $(W|K,\nu)$. We now consider the diagram,
	\[ \begin{tikzcd}[row sep= 2ex, column sep=0.8em]
	K(\a) \ar[r, symbol=\subset] & K(\a,\th) \ar[r, symbol=\subset] &W(\a)  \\
	K \ar[u, symbol=\subset] \ar[r, symbol=\subset] &K(\th) \ar[u, symbol=\subset] \ar[r, symbol=\subset] &W  \ar[u, symbol=\subset] \\
	\end{tikzcd}. \]
	$(W|K(\th),\nu)$ is a purely wild extension and $(K(\th,\a)|K(\th),\nu)$ is a defectless unramified extension. So the extensions are valuation disjoint. From Lemma \ref{Lemma valn disj tame and purely wild} we observe that $(W(\a)|W,\nu)$ is a defectless unramified extension and $[W(\a):W] = [K(\th,\a):K(\th)]$. Thus $\a\notin W$. Since this is true for any $\a$ such that $g(\a)=0$, we conclude that any conjugate of $\a$ also does not lie in $W$.
	\pars Suppose that $(W,\nu)$ is a unique maximal purely wild extension up to isomorphisms over $K$. Since $(K(\a)|K,\nu)$ is a purely wild extension, it is embedded in some maximal purely wild extension. Hence there exists $\o\in\Aut(K^\ac|K)$ such that $\o\a\in W$, which is a contradiction. Thus $(K,\nu)$ does not admit a unique maximal purely wild extension. 
\end{proof}

Suppose that in the setting of the previous theorem, we have that $(W|K,\nu)$ is a maximal immediate algebraic extension and $(K(\th)|K,\nu)$ is a minimal immediate subextension. Since immediate extensions are purely wild, the extensions $(K(\a,\th)|K(\th),\nu)$ and $(K(\a,\th)|K(\a),\nu)$ are still defectless and unramified extensions. From the multiplicative property of the defect, we obtain that
\[ [K(\th):K] = d(K(\th)|K,\nu) = d(K(\th,\a)|K,\nu) = d(K(\a)|K,\nu) \leq [K(\a):K] =[K(\th):K]. \]
Hence $(K(\a)|K,\nu)$ is an immediate extension. We have now arrived at the following result.

\begin{Theorem}
	Let notations and assumptions be as in Theorem \ref{Theorem non-unique eq char}. Also suppose that $(K(\th)|K,\nu)$ is a minimal immediate extension. Then $(K,\nu)$ does not admit a unique maximal immediate algebraic extension. 
\end{Theorem} 

If the extension $W|K$ is separable, then the previous two theorems are equivalent in light of the next proposition. 

\begin{Proposition}
	Let $(K,\nu)$ be a henselian valued field with $\ch K=p>0$. Suppose that $(K,\nu)$ admits a nontrivial maximal purely wild extension $(W|K,\nu)$ such that $W|K$ is separable. Then $K$ is perfect. Consequently the maximal purely wild extensions of $(K,\nu)$ correspond to the maximal immediate algebraic extensions of $(K,\nu)$.  
\end{Proposition}

\begin{proof}
	The first assertion is immediate since purely inseparable extensions of valued fields are purely wild. So $K$ is perfect and hence $\nu K = \frac{1}{p^\infty} \nu K$ and $K\nu = K\nu^{\frac{1}{p^\infty}}$. Now any maximal purely wild extension $(M|K,\nu)$ satisfies that $\nu M = \frac{1}{p^\infty} \nu K$ and $M\nu = K\nu^{\frac{1}{p^\infty}}$. Thus $(M|K,\nu)$ is immediate. Conversely any purely wild extension $(L|K,\nu)$ satisfies the fact that $\nu L/\nu K$ is a $p$-group and $L\nu|K\nu$ is purely inseparable. Thus $(L|K,\nu)$ is immediate. The result now follows. 
\end{proof}

\parb We now shift our attention to henselian valued fields of mixed characteristic. We first investigate purely wild extensions $(K(\th)|K,\nu)$ where the minimal polynomial of $\th$ over $K$ is of the form $\mathfrak{A}(X)-c$ as presented in Theorem \ref{Theorem non-unique eq char}. The Frobenius endomorphism holds over fields of positive characteristic which accounts for the additivity of $\mathfrak{A}(X)$ when $\ch K >0$. Even though it fails to hold when $\ch K=0$, we observe that for sufficiently large values of the generator $\th$, we can obtain results analogous to Theorem \ref{Theorem non-unique eq char}.

\begin{Lemma}\label{Lemma extns by p-polynomials mixed char}
	Let $(K,\nu)$ be a henselian valued field with $\ch K=0$ and $\ch K\nu=p>0$. Let $\mathfrak{A}(X) = \sum_{i=0}^{n} a_i X^{p^i} \in \mathfrak{O}_K[X]$. Suppose that $f(X):=\mathfrak{A}(X)-c \in K[X]$ has no roots in $K$. Let $f(\th)=0$. Further suppose that $\nu a_0=\nu a_n=0$ and $\nu\th \geq -\frac{\nu p}{p^n}$. Then the following statements hold true:
	\sn (1) $\nu(\th-K)\leq 0$.
	\n (2) If $\nu(\th-b)=0$ for some $b\in K$ then $(K(\th)|K,\nu)$ is a defectless unramified extension with $K(\th)\nu =K\nu((\th-b)\nu)$.
	\n (3) If $\nu c=0$ then $\nu\th=0$. 
\end{Lemma}

\begin{proof}
	We have that $\sum_{i=0}^{n} a_i \th^{p^i}=c$ where $a_i \in \mathfrak{O}_K$. So, $\nu\th>0\Longrightarrow \nu c>0$. Consequently, $f\nu(X) = \mathfrak{A}\nu (X)$. The fact that $a_0\nu\neq 0$ then implies that $f\nu(X)$ admits $0$ as a simple root. The field $(K,\nu)$ being henselian then implies that $f(X)$ admits a simple root over $K$, thus yielding a contradiction. So $\nu\th\leq 0$. When $\nu\th<0$, we observe that $\nu \sum_{i=0}^{n} a_i \th^{p^i} = p^n \nu\th$. Hence, $\nu c=0\Longrightarrow\nu\th=0$. 
	
	\pars We now suppose that $\nu(\th-b)>0$ for some $b\in K$. So $\nu\th = \nu b$. Define $f_b(X) := f(X+b)= \mathfrak{A}(X+b)-c$. Thus, 
	\[ f_b(X) = \mathfrak{A}(X) + \sum_{i=0}^{n} a_i \sum_{k=1}^{p^i-1} \binom{p^i}{k} X^{p^i-k}b^k -(c-\mathfrak{A}(b)). \]
	Now, the fact that $f(\th)=0$ implies that $f_b(\th-b)=0$. So, 
	\[ \mathfrak{A}(\th-b) + \sum_{i=0}^{n} a_i \sum_{k=1}^{p^i-1} \binom{p^i}{k} (\th-b)^{p^i-k}b^k = c-\mathfrak{A}(b).   \]
	We know that $p$ divides $\binom{p^i}{k}$ for each $1\leq k \leq p^i-1$. Hence, $\nu\binom{p^i}{k}b^k \geq \nu p +k\nu\th \geq \nu p- \frac{k}{p^n}\nu p >0$. Again, $\nu(\th-b)>0 \Longrightarrow \nu\mathfrak{A}(\th-b)>0$. Consequently we obtain that $\nu(c-\mathfrak{A}(b)) >0$. Thus $f_b\nu (X) = \mathfrak{A}\nu (X)$ which admits $0$ as a simple root. It follows that $f_b(X)$ admits a simple root over $K$. The fact that $f_b(X) = f(X+b)$ then yields a contradiction. So, $\nu(\th-K)\leq 0$. 
	
	\pars Now suppose that $\nu(\th-b)=0$ for some $b\in K$. If $\nu\th<0$, then $\nu\th=\nu b <0$. The lemma now follows from the same arguments as in Lemma \ref{Lemma extns by p-polynomials}. If $\nu\th=0$, then $\nu c=0$ by the preceding arguments. Thus $f\nu(X)=\mathfrak{A}\nu(X)-c\nu \in K\nu[X]$. The lemma again follows from similar arguments as in Lemma \ref{Lemma extns by p-polynomials}. 
\end{proof}

\begin{Theorem}\label{Theorem non-unique mixed char}
	Let $(K(\th)|K,\nu)$ be a purely wild extension of henselian valued fields with $\ch K=0$ and $\ch K\nu=p>0$. Suppose that the minimal polynomial of $\th$ over $K$ is of the form $f(X):=\mathfrak{A}(X)-c$, where,
	\[ \mathfrak{A}(X):= \sum_{i=0}^{n} a_i X^{p^i}\in\mathfrak{O}_K[X]. \]
	Suppose that $\nu a_0=\nu a_n=0$, $\nu\th> - \frac{\nu p}{p^n}$ and that $K\nu$ is not $\mathfrak{A}\nu$-closed. Then $(K,\nu)$ does not admit a unique maximal purely wild extension. If $(K(\th)|K,\nu)$ is an immediate extension, then $(K,\nu)$ does not admit a unique maximal immediate algebraic extension.
\end{Theorem}

\begin{proof}
	Let $t\in K$ with $\nu t=0$ such that $\mathfrak{A}\nu(X)-t\nu$ does not have any roots in $K\nu$. Define the polynomials $g(X):= \mathfrak{A}(X)-(c+t)$ and $h(X):= \mathfrak{A}(X)-t$. Then $h\nu(X)$ does not admit any roots in $K\nu$. The fact that $a_0\nu\neq 0$ implies that $h\nu(X)$ is separable over $K\nu$. We have that the residue field extension $K(\th)\nu|K\nu$ is purely inseparable. It follows that $h\nu(X)$ does not admit any roots in $K(\th)\nu$.
	\pars Let $\a\in K^\ac$ such that $g(\a)=0$. It follows from Lemma \ref{Lemma extns by p-polynomials mixed char} that $\nu(\th-K) <0$. Further, $\nu c = p^n\nu\th < 0$. Thus $\nu(c+t)=\nu c < 0$ and consequently $\nu\a=\nu\th$. Now, 
	\begin{align*}
	\mathfrak{A}(\a-\th)&= \mathfrak{A}(\a)-\mathfrak{A}(\th)+ \sum_{i=0}^{n} a_i \sum_{k=1}^{p^i-1} \binom{p^i}{k} (-1)^k \a^{p^i-k}\th^k \\
	&= \sum_{i=0}^{n} a_i \sum_{k=1}^{p^i-1} \binom{p^i}{k} (-1)^k \a^{p^i-k}\th^k + t.
	\end{align*} 
	For each $1\leq k \leq p^i-1$, we have that, $\nu \binom{p^i}{k} \a^{p^i-k}\th^k \geq \nu p +p^i \nu\th > \nu p - \frac{p^i}{p^n} \nu p \geq 0$. Thus, 
	\[ \nu\mathfrak{A}(\a-\th)=0 \text{ and hence } \nu(\a-\th)=0.  \]
	Further, $\mathfrak{A}(\a-\th)\nu = \mathfrak{A}\nu((\a-\th)\nu) = t\nu$, that is, $(\a-\th)\nu$ is a root of $h\nu$.
	\pars Now suppose that $g(X)$ admits a root $\b\in K(\th)$. The preceding arguments imply that $(\b-\th)\nu$ is a root of $h\nu$, which contradicts the fact that $h\nu$ does not admit any roots in $K(\th)\nu$. So $g(X)$ does not have any roots in $K(\th)$. Lemma \ref{Lemma extns by p-polynomials mixed char} now implies that $(K(\a,\th)|K(\th),\nu)$ is a defectless and unramified extension with $K(\a,\th)\nu = K(\th)\nu((\a-\th)\nu)$. Again, observe that $h\nu$ being a separable polynomial over $K\nu$ implies that $(\a-\th)\nu$ is separable over $K\nu$. Thus $K\nu(\a-\th)\nu$ and $K(\th)\nu$ are linearly disjoint over $K\nu$. It follows that,
	\[ [K(\a,\th):K(\th)] = [K\nu((\a-\th)\nu) : K\nu]. \]
	
	\pars Now suppose that $f(X)$ admits a root over $K(\a)$. The polynomial $f(X)$ being irreducible over $K$ implies that each root of $f$ is a $K$-conjugate of $\th$. Further, the henselianity of $(K,\nu)$ implies that each conjugate has the same value. So without any loss of generality, we can assume that $\th\in K(\a)$. So $p^n=[K(\th):K]\leq[K(\a):K]\leq p^n$. It follows that $K(\a)=K(\th)$. which implies that $h\nu(X)$ has the root $(\a-\th)\nu$ in $K(\th)\nu$, thus yielding a contradiction. So $f(X)$ does not have any roots in $K(\a)$. From Lemma \ref{Lemma extns by p-polynomials mixed char} it now follows that $(K(\a,\th)|K(\a),\nu)$ is a defectless and unramified extension with $K(\a,\th)\nu = K(\a)\nu((\a-\th)\nu)$.
	\pars Observe that, 
	\begin{align*}
	[K(\a,\th):K(\a)] = [K(\a)\nu((\a-\th)\nu): K(\a)\nu] &\leq [K\nu((\a-\th)\nu):K\nu] = [K(\a,\th):K(\th)].\\
	[K(\a):K] &\leq [K(\th):K].
	\end{align*}
	It follows that both the above inequalities are equalities. Using the same arguments as presented in Theorem \ref{Theorem non-unique eq char}, it now follows that $(K(\a)|K,\nu)$ is a purely wild extension and that $(K,\nu)$ does not admit a unique maximal purely wild extension.
\end{proof}

\begin{Corollary}
	Let $(K(\th)|K,\nu)$ be a purely wild extension of henselian valued fields with $\ch K=0$ and $\ch K\nu=p>0$. Suppose that the minimal polynomial of $\th$ over $K$ is of the form $X^p-X-c$. Suppose that $\nu\th> - \frac{\nu p}{p}$ and that the residue field $K\nu$ is not Artin-Schreier closed. Then $(K,\nu)$ does not admit a unique maximal purely wild extension. If $(K(\th)|K,\nu)$ is an immediate extension, then $(K,\nu)$ does not admit a unique maximal immediate algebraic extension.  
\end{Corollary}

\begin{Remark}\label{Remark splitting field additive polynomials}
	We make a quick remark regarding splitting fields of polynomials of the form $\mathfrak{A}(X)$ as mentioned above, irrespective of the characteristic of the base field. Let $(K,\nu)$ be a henselian valued field with $\ch K\nu=p>0$. Let $\mathfrak{A}(X):= \sum_{i=0}^{n} a_i X^{p^i} \in \mathfrak{O}_K[X]$ and suppose that $\nu a_0 = \nu a_n = 0$. Then $\mathfrak{A}(X)$ and $\mathfrak{A}\nu(X)$ are separable polynomials and $\deg \mathfrak{A} = \deg \mathfrak{A}\nu$. Let $\g\in K^\ac\setminus\{0\}$ be such that $\mathfrak{A}(\g)=0$. Then, $\sum_{i=0}^{n} a_i \g^{p^i}=0 \Longrightarrow \nu\g\geq 0$. Again we observe that $0$ is a simple root of $\mathfrak{A}$ and of $\mathfrak{A}\nu$. Hence we have that $\nu\g=0$. Now from Remark \ref{Remark hensels lemma} we have a decomposition $\mathfrak{A} = f_1 \dotsc f_r$ such that $\mathfrak{A}\nu = f_1\nu \dotsc f_r\nu$ is an irreducible decomposition of $\mathfrak{A}\nu$, and $\deg f_i = \deg f_i \nu$ for each $i$. With the same arguments as in Lemma \ref{Lemma extns by p-polynomials}, we observe that $(K(\g)|K,\nu)$ is a defectless and unramified extension. Let $E$ denote the splitting field of $\mathfrak{A}(X)$ over $K$. Then $E$ is the compositum of all such fields $K(\g)$ where $\mathfrak{A}(\g)=0$. Hence $(E|K,\nu)$ is a defectless and unramified extension.  
\end{Remark}

We now study purely wild Kummer extensions $(K(\th)|K,\nu)$. Recall that a Kummer extension $K(\th)|K$ is a degree $p$ Galois extension where the minimal polynomial of $\th$ over $K$ is of the form $X^p-a$. We observe that when the Kummer generator has value $0$, we can obtain a bound on the set of values $\nu(\th-K)$. 

\begin{Lemma}\label{Lemma value group bound Kummer extn}
	Let $(K,\nu)$ be a henselian valued field with $\ch K=0$ and $\ch K\nu=p>0$. Suppose that $K$ contains a primitive $p$-th root of unity. Let $K(\th)|K$ be a Kummer extension and suppose that $\nu \th = 0$. Then $\nu(\th-K) \leq \frac{\nu p}{p-1}$.   
\end{Lemma}

\begin{proof}
	Suppose that $\nu(\th - t) > \frac{\nu p}{p-1}$ for some $t \in K$. So $\nu \th = \nu t = 0$ and hence, $\nu(\th-K) = \nu(\frac{\th}{t} - K)$. Further $\frac{\th}{t}$ is also a Kummer generator of $K(\th)|K$ and $\nu(\frac{\th}{t}) = 0$. Thus we can assume that $t = 1$. So we now have $\nu(\th-1) > \frac{\nu p}{p-1}$. We consider the binomial expansion of $(\th - 1)^p$. Since $\binom{p}{i}$ is a multiple of $p$ for each $i = 1, \dotsc, p-1$, we have that,
	\[ (\th-1)^p = (\th^p - 1) + p f(\th) \]
	where $f(\th) \in \mathcal{O}_K [\th] \subseteq \mathcal{O}_{K(\th)}$. Again $\th - 1$ divides $f(\th)$ in $\mathcal{O}_K [\th]$. Hence we have, 
	\begin{equation}\label{Kummer extn binom form}
	(\th-1)^p = (\th^p - 1) + p (\th-1)g 
	\end{equation}
	where $\nu g \geq 0$. We observe that $\nu(\th-1)^p = p\nu(\th-1) > \frac{p}{p-1} \nu p$ and $\nu[p(\th-1)g] > \frac{p}{p-1} \nu p$. Thus we have $\nu(\th^p - 1) > \frac{p}{p-1} \nu p$. Let $\th^p = a \in K$. Then from Lemma \ref{lemma 1+a = b^p} we have that there exists $b \in K$ such that $b^p = a$. But this implies that $\th \in K$ which is a contradiction. Hence we have the lemma. 	
\end{proof}

We are now ready to give a complete ramification theoretic classification of a Kummer extension $(K(\th)|K, \nu)$. 

\begin{Theorem}\label{Theorem classification kummer}
	Let $(K,\nu)$ be a henselian valued field with $\ch K=0$ and $\ch K\nu=p>0$. Suppose that $K$ contains a primitive $p$-th root of unity. Let $K(\th)|K$ be a Kummer extension. Then the following cases are possible. 
	\sn (I) $\nu(\th - K)$ does not have a maximal element. Then $(K(\th)|K, \nu)$ is immediate. 
	\n (II) $\nu(\th-K)$ attains a maximal element $\g$. These are the following subcases. 
	\n \hspace{5 mm} (II.A) $\g \notin \nu K$. Then $(K(\th)|K, \nu)$ is purely wild and defectless. 
	\n \hspace{5 mm} (II.B) $\g \in \nu K$. We can choose $\th$ such that $\nu \th = 0$. We have $0 \leq \g \leq \frac{\nu p}{p-1}$.
	\n \hspace{10 mm} (II.B.1) $0 \leq \g < \frac{\nu p}{p-1}$. Then $(K(\th)|K, \nu)$ is purely wild and defectless. 
	\n \hspace{10 mm} (II.B.2.) $\g = \frac{\nu p}{p-1}$. Then $(K(\th)|K, \nu)$ is defectless and unramified and $K(\th)\nu|K\nu$ is an Artin-Schreier extension.
\end{Theorem}

\begin{proof}
	Case I is a direct application of Lemma 2.21 of [\ref{Kuh A-S extensions and defectless fields paper}]. Now suppose that $\nu(\th-K)$ admits a maximal value $\g$. If $\g \notin \nu K$, then $\nu K$ is a proper subgroup of $\nu K(\th)$. Hence from the Lemma of Ostrowski we obtain that $[\nu K(\th) : \nu K] = p$, and thus $(K(\th)|K, \nu)$ is purely wild and defectless. 
	\pars Now we suppose that $\g \in \nu K$. So by Lemma \ref{lemma distance from K} we have that $\nu(\th-K) \subseteq \nu K$, and in particular $\nu \th \in \nu K$. Let $\nu \th = \nu b$ for some $b \in K$. Then $\frac{\th}{b}$ is another Kummer generator of $K(\th)|K$ with $\nu(\frac{\th}{b}) = 0$. For any $t \in K$, we have that $\nu(\frac{\th}{b} - t) = \nu(\th - bt) - \nu b \in \nu K$. Thus $\nu(\frac{\th}{b} - K) \subseteq \nu K$ and it attains the maximal value $\g-\nu b \in \nu K$. So we can assume that $\nu \th = 0$. From Lemma \ref{Lemma value group bound Kummer extn} we then have that $0 \leq \g \leq \frac{\nu p}{p-1}$. For the rest of the proof, we will work with this setup. 
	\pars Now suppose that $0 \leq \g < \frac{\nu p}{p-1}$. Let $\nu(\th-t) = \g = \nu b$ for some $b, \, t \in K$. Analogous to equation (\ref{Kummer extn binom form}), we have an expression, 
	\[ (\th-t)^p = (a - t^p) + p(\th-t)g \]
	where $a = \th^p$ and $\nu g \geq 0$. Then, 
	\[ \frac{(\th-t)^p}{b^p} - \frac{(a-t^p)}{b^p} = \frac{p(\th-t)g}{b^p}. \]
	Now $\nu(\frac{p(\th-t)g}{b^p}) \geq \nu p + \g - p \g > 0$ and hence $(\frac{\th-t}{b} \nu)^p = \frac{a-t^p}{b^p} \nu \in K \nu$. If $\frac{\th-t}{b} \nu \in K \nu$, say, $\frac{\th-t}{b} \nu = c\nu$ for some $c \in K$, then it implies that $\nu(\th-t-bc) > \nu b = \g$ which contradicts the maximality of $\g$. Hence $(\frac{\th-t}{b}) \nu \notin K \nu$. Since $\ch K \nu = p$, we have that $K\nu (\frac{\th-t}{b} \nu) |K \nu$ is purely inseparable of degree $p$. Thus $(K(\th)|K, \nu)$ is a purely wild extension with $\nu K(\th) = \nu K$ and $K(\th) \nu = K \nu (\frac{\th-t}{b} \nu)$.
	\pars We now consider the final case, when $\g = \frac{\nu p}{p-1}$. Let $\g = \nu(\th-t)$ for some $t \in K$. We first observe that we can assume that $t = 1$. Indeed $\nu \th = 0$ and $\nu(\th-t) = \g>0$ implies that $\nu \th = \nu t = 0$. So $\nu(\frac{\th}{t} - K) = \nu(\th-K)$. So replacing $\th$ by $\frac{\th}{t}$, we can assume that $t = 1$. Hence we now have that $\nu(\th-1) = \g$. Now from Lemma \ref{lemma C^{p-1} = p}, there exists $C \in K$ such that $C^{p-1} = -p$. Let $\th_C := \frac{\th-1}{C}$. Thus $\th_C \in K(\th)$ and $\nu \th_C = 0$. We will now use a transformation mentioned in [\ref{Kuh gdr}]. The minimal polynomial of $\th$ over $K$ is $f(X) = X^p - a$. Substituting $X = CY+1$ and dividing by $C^p$, we arrive at the polynomial
	\[ g(Y) = Y^p + m(Y) - Y - \frac{a-1}{C^p} \]
	where $m(Y) \in \mathcal{M}_K [Y]$. Observe further that $f(x) = 0$ if and only if $g(\frac{x-1}{C}) = 0$. In particular, $g(\th_C) = 0$. Now $K(\th) = K(\th_C)$ and hence $g(Y)$ is the minimal polynomial of $\th_C$ over $K$. Again, we have the expression, 
	\[ (\th-1)^p = (a-1) + p(\th-1)g \]
	where $\nu g \geq 0$. Thus $\nu(a-1) \geq \frac{p}{p-1} \nu p$. If $\nu(a-1) > \frac{p}{p-1} \nu p$, then Lemma \ref{lemma 1+a = b^p} would imply that $\th \in K$. So $\nu(a-1) = \frac{p}{p-1} \nu p$ and hence $0 \neq \frac{a-1}{C^p} \nu \in K \nu$.
	\pars Consider the residue polynomial $g \nu (Y) = Y^p-Y- \frac{a-1}{C^p} \nu \in K \nu [Y]$. Thus $g \nu (Y)$ is an Artin-Schreier polynomial, hence it either splits completely over $K \nu$, or is irreducible. If $g \nu$ splits over $K \nu$, then by Hensel's Lemma, $g(Y)$ splits completely over $K$ which is a contradiction to the irreducibility of $g$ over $K$. So $g \nu$ is irreducible over $K \nu$, and is thus the minimal polynomial of $\th_C \nu$ over $K \nu$. So $K(\th)\nu = K \nu (\th_C \nu)$ is an irreducible Artin-Schreier extension over $K \nu$. Hence $\nu K(\th) = \nu K$ and $(K(\th)|K, \nu)$ is a defectless unramified extension. 
\end{proof}

If $(K,\nu)$ is a henselian valued field containing a primitive $p$-th root of unity such that the residue field $K\nu$ is not Artin-Schreier closed and $(K(\th)|K,\nu)$ is a purely wild Kummer extension, we can illustrate the non-uniqueness of maximal purely wild extension of $(K,\nu)$ by a constructive argument similar to the one in Theorem \ref{Theorem non-unique eq char}. Let $\th^p=a\in K$ and let $X^p-X-t\nu$ be an irreducible Artin-Schreier polynomial over $K\nu$. From Lemma \ref{lemma C^{p-1} = p} we have that there is an element $C \in K$ such that $C^{p-1} = -p$. Define $\a\in K^\ac$ such that $\a^p = a(1+C^p t)$. Then from the ultrametric inequality we have that, 
\[ \nu\frac{\a^p}{\th^p} = \nu(1+C^pt) = 0 \Longrightarrow \nu\a=\nu\th.   \]
If $\nu\th\notin \nu K$, then $\nu\a\notin\nu K$ and hence $(K(\a)|K,\nu)$ is also a purely wild defectless extension. Now suppose that $\nu\a=\nu\th=0$. From Lemma \ref{Lemma value group bound Kummer extn} we have that $\nu(\th - K) < \frac{\nu p}{p-1}$. Suppose that $\nu(\a - \th) < \frac{\nu p}{p-1} = \nu C$. Now we have an expression, 
\[ (\a - \th)^p = (\a^p - \th^p) + p(\a - \th) g = a t C^p + p(\a-\th)g \]
where $\nu g \geq 0$. We have that $\nu(at C^p) = p \nu C > \nu(\a-\th)^p$. Hence, 
\[ \nu(\a-\th)^p = \nu p + \nu(\a - \th) + \nu g \geq \nu p +\nu (\a -\th). \]
But this implies that $\nu(\a -\th) \geq \frac{\nu p}{p-1}$ which contradicts our assumption. So, 
\[ \nu(\a -\th) \geq \frac{\nu p}{p-1} > \nu(\th -K). \]
We thus obtain that for any $d \in K$, we have that $\nu(\th - d) = \nu(\a -d)$. In particular, we have that $\nu(\th -K) = \nu(\a -K)$ and hence $(K(\a)|K, \nu)$ is a purely wild extension. The fact that they both are degree $p$ extensions over $K$ implies that either $K(\th)=K(\a)$ or that they are linearly disjoint over $K$. Let us define $f(X):= X^p - (1+C^pt) \in K[X]$. Considering the transformation $X = CY+1$ and dividing by $C^p$, we obtain the polynomial $g(Y) = Y^p + m(Y) - Y - t \in K[Y]$ where $m(Y) \in \mathcal{M}_K[Y]$. The residue polynomial is $g \nu (Y) = Y^p-Y-t\nu \in K\nu [Y]$. We also observe that $f(x) = 0$ if and only if $g(\frac{x-1}{C}) = 0$. Thus, 
\[ f(\frac{\a}{\th}) = 0\Longrightarrow g(\eta) =0 \text{ where } \eta = \frac{\frac{\a}{\th}-1}{C} \in K(\th, \a). \]
\pars Note that $\nu\frac{\a}{\th} = 0$. We then have the expression, 
\[ (\frac{\a}{\th} -1)^p = C^p t + p(\frac{\a}{\th}-1)h \]
where $\nu h \geq 0 $. If $\nu(\frac{\a}{\th}-1) < \frac{\nu p}{p-1} = \nu C$, then $p \nu (\frac{\a}{\th}-1) = \nu [p (\frac{\a}{\th}-1)h] \geq \nu p + \nu(\frac{\a}{\th}-1)$ which again leads to a contradiction. So $\nu(\frac{\a}{\th}-1) \geq \nu C$, that is, $\nu \eta \geq 0$. Hence $\eta \nu \in K(\th,\a) \nu$. Then the fact that $g(\eta) = 0$ implies that $g \nu (\eta \nu) = 0$. So we have that $\eta \nu \neq 0$ and thus $\nu (\frac{\a}{\th}-1) = \nu C = \frac{\nu p}{p-1}$. 
\newline Now if $K(\th) = K(\a)$, then $\eta \in K(\a)$. Hence, $\eta \nu \in K(\a) \nu$ is a root of $g \nu$, that is, $g\nu$ splits completely over $K(\a)\nu$. The fact that $K(\a)\nu|K\nu$ is purely inseparable then implies that $g\nu$ splits completely over $K\nu$. But this contradicts the irreducibility of $g \nu$ over $K\nu$. So, $K(\th) \neq K(\a)$, and thus, $K(\th)$ and $K(\a)$ are linearly disjoint over $K$.  
\pars We now consider the Kummer extension $(K(\a, \th)|K(\a), \nu)$ with the Kummer generator $\frac{\a}{\th}$ with $\nu \frac{\a}{\th} =0$. We have observed that $\nu (\frac{\a}{\th}-1) = \frac{\nu p}{p-1}$. From Theorem \ref{Theorem classification kummer} we then obtain that $(K(\a, \th)|K(\a), \nu)$ is a defectless and unramified extension. Similarly the extension $(K(\a, \th)|K(\th), \nu)$ is also a defectless and unramified extension. The same arguments as in Theorem \ref{Theorem non-unique eq char} now yields that $(K,\nu)$ does not admit a unique maximal purely wild extension.

\begin{Corollary}
	Let $(K,\nu)$ be a henselian valued field with $\ch K = 0$ and $\ch K\nu = p>0$. Suppose that $K$ contains a primitive $p$-th root of unity. Let $\th\in K^\ac$ such that the minimal polynomial of $\th$ over $K$ is given by $X^{p^n}-a \in K[X]$ and suppose that $\nu\th=0$. Then the following statements hold true:
	\sn (I) $\nu(\th-K) \leq \frac{\nu p}{p-1}$.
	\n (II) $\frac{\nu p}{p-1} \in \nu (\th-K)$ if and only if $(K(\th)|K,\nu)$ is a defectless and unramified Kummer extension. 
	\n (III) Suppose that $(K(\th)|K,\nu)$ is purely wild. Then $(K,\nu)$ does not admit a unique maximal purely wild extension if the residue field $K\nu$ is not Artin-Schreier closed. 
\end{Corollary}

\begin{proof}
	Define $\th_{n-1} := \th^{p^{n-1}}$. Then $\th_{n-1}^p= a\in K$ and hence $[K(\th_{n-1}):K] \leq p$. The fact that $\th_{n-1} \in K(\th)$ implies that $[K(\th_{n-1}):K]$ divides $[K(\th):K]=p^n$. Hence, $[K(\th_{n-1}):K] = 1$ or $p$. If $\th_{n-1} = \th^{p^{n-1}} \in K$, then $[K(\th):K] \leq p^{n-1}$ which yields a contradiction. So $[K(\th_{n-1}):K] = p$ and hence $(K(\th_{n-1})|K,\nu)$ is a Kummer extension. Further, the fact that $\nu\th=0$ implies that $\nu\th_{n-1} = 0$. Hence from Lemma \ref{Lemma value group bound Kummer extn} we have that $\nu(\th_{n-1}-K) \leq \frac{\nu p}{p-1}$. 
	\newline Now suppose that $\nu(\th-t) > \frac{\nu p}{p-1}$ for some $t\in K$. Then $\nu\th = \nu t = 0$. The minimal polynomial of $\frac{\th}{t}$ over $K$ is given by $X^{p^n} - \frac{a}{t^{p^n}} \in K[X]$. Hence we can assume that $\nu (\th-1) > \frac{\nu p}{p-1}$. From the additive property of valuation we observe that $\nu(\th_{n-1} - 1) \geq \nu(\th-1)> \frac{\nu p}{p-1}$ which yields a contradiction. Hence, $\nu(\th-K) \leq \frac{\nu p}{p-1}$.   
	\pars If $(K(\th)|K, \nu)$ is a defectless and unramified Kummer extension, then $\frac{\nu p}{p-1} \in \nu(\th-K)$ [Theorem \ref{Theorem classification kummer}]. Conversely, suppose that $\frac{\nu p}{p-1} \in\nu(\th-K)$. Without any loss of generality we can again assume that $\nu(\th-1) = \frac{\nu p}{p-1}$. From the preceding discussions it follows that $\nu(\th_{n-1} -1) = \frac{\nu p}{p-1}$ and hence $(K(\th_{n-1})|K,\nu)$ is a defectless and unramified Kummer extension. We have that $\th_{n-1}^p = a$. From the proof of Theorem \ref{Theorem classification kummer} it follows that $\nu(a-1) = \frac{p}{p-1}\nu p$. Now consider the extension $K(\th)|K(\th_{n-1})$ with the minimal polynomial $X^{p^{n-1}}-\th_{n-1}$. If the extension is non-trivial, we can construct a Kummer extension $K(\th_{n-2})|K(\th_{n-1})$ where $\th_{n-2} : = \th^{p^{n-2}}$. Using the same arguments as above, we have that $\nu(\th_{n-2}-1) = \frac{\nu p}{p-1}$. Further, the fact that $\th_{n-2}^p = \th_{n-1}$ implies that $\nu(\th_{n-1}-1) = \frac{p}{p-1}\nu p$ which leads to a contradiction. Hence $K(\th_{n-1}) = K(\th)$ and thus we obtain (II).
	\pars Suppose that $(K(\th)|K,\nu)$ is a purely wild extension. Then $(K(\th_{n-1})|K,\nu)$ is a purely wild Kummer extension and the result now follows.
\end{proof}


\section{Group theoretic conditions for non-uniqueness} \label{Section group theoretic non-unique}

Let us consider the situation of Theorem \ref{Theorem non-unique eq char}. Let $f(X):= \mathfrak{A}(X)-c$ be the minimal polynomial of $\th$ and let $\mathfrak{A}\nu(X)-t\nu$ does not have any roots in $K\nu$. Then $h(X):= \mathfrak{A}(X)-t$ does not have any roots in $K$. Let $E$ denote the splitting field of $\mathfrak{A}(X)$ over $K$. We have observed in Remark \ref{Remark splitting field additive polynomials} that $(E|K,\nu)$ is a Galois defectless unramified extension. The additivity of $\mathfrak{A}(X)$ implies that all the conjugates of $\th$ are of the form $\th+\g$, where $\mathfrak{A}(\g)=0$. Setting $L:= K(\th)$, we observe that $(L|K,\nu)$ is a purely wild extension such that $L.E=E(\th)$ is the splitting field of $f(X)$ over $K$, and hence is a Galois extension over $K$. Further, let $\b$ be a root of $h(X)$ over $K$. It follows from Lemma \ref{Lemma extns by p-polynomials} that $F:= K(\b)$ is a defectless and unramified extension over $(K,\nu)$, and $F.E|K$ is a Galois extension. In this section, we show that under certain additional group theoretic assumptions, we can again achieve non-uniqueness irrespective of the characteristic of $K$.

The following lemmas will be required in the proof of Theorem \ref{Theorem non-unique}.

\begin{Lemma}\label{Lemma surjective homo subgroup}
	Let $\phi:G\longrightarrow \G$ be a homomorphism of groups with $K:=\ker\phi$. Suppose that there exists a subgroup $H \leq G$ such that $\phi(H)=\phi(G)$. Then $G=K.H$.
\end{Lemma}

\begin{proof}
	The assertion implies that for each $g\in G$, there exists $h\in H$ such that $\phi(g)=\phi(h)$. Hence $k:=gh^{-1}\in K$, that is, $g=kh$. Hence we have the lemma.  
\end{proof}

\begin{Lemma}\label{Lemma semidirect product}
	Let $G=H\rtimes N$ be a semidirect product of groups where $N\trianglelefteq G$. Let $T$ be a subgroup of $G$ such that $N\subseteq T$. Then $T= (H\sect T)\rtimes N$. 
\end{Lemma}

\begin{proof}
	Let $t\in T$. We have an expression $t=hn$ where $h\in H$ and $n\in N$. Then the fact that $N\subseteq T$ implies that $h\in H\sect T$. The lemma now follows.
\end{proof}

\begin{Theorem}\label{Theorem non-unique}
	Let $(K,\nu)$ be a henselian valued field with $\ch K\nu=p>0$. Let $(L|K,\nu)$ be a separable nontrivial purely wild extension and $(F|K,\nu)$ be a nontrivial tame extension. Suppose that $(E|K,\nu)$ is an abelian Galois tame extension such that $L.E|K$ and $F.E|K$ are Galois extensions. Further suppose that the following conditions are satisfied:
	\sn (1) $\th: \Gal(F.E|K) \longrightarrow \Gal(L.E|K)$ is a continuous isomorphism,
	\n (2) $\th_{\res} : \Gal(F.E|E) \longrightarrow \Gal(L.E|E)$ is a continuous isomorphism. 
\newline Let $(L.E|K)^\ab$ denote the maximal abelian Galois subextension of $L.E|K$. Then the following cases are possible:
\sn (a) $E\subsetneq (L.E|K)^\ab$. Then $(K,\nu)$ does not admit a unique maximal purely wild extension. 
\n (b) $E=(L.E|K)^\ab$. Then $(E,\nu)$ does not admit a unique maximal purely wild extension. 
\end{Theorem}

\begin{proof}
	Set $G:= \Gal(K^\sep|K)$, $H:= \Gal(K^\sep|L)$, $M:= \Gal(K^\sep|F)$ and $N:= \Gal(K^\sep|E)$. Then $H$ and $M$ are closed subgroups of $G$ and $N$ is a closed normal subgroup of $G$ such that $G/N$ is abelian. From Galois correspondence we have that $G_r\subseteq N$ and $G_r\subseteq M$. The fact that $(L|K,\nu)$ is purely wild implies that $L\sect K^r = K$. Hence $G_r.H=G$ and thus $N.H=G$. Further, $H\sect N$ and $M\sect N$ are closed normal subgroups of $G$. Set $\G := \frac{G}{H\sect N} = \Gal(L.E|K)$. Let $\D\subseteq \G$ be the closure of the commutator subgroup $[\G,\G]$. Then $\D$ is a closed normal subgroup of $\G$ [Proposition 1.5, Chapter 6, \ref{Karpilovsky}]. Also,
	\[ \Gal(L.E| (L.E|K)^\ab) = \D, \, \, \G^{\ab} := \Gal((L.E|K)^\ab|K) = \G/\D. \]
	The hypotheses of the theorem translate to the following group theoretic relations:
	\begin{align*}
		\th : \frac{G}{M\sect N} &\iso \frac{G}{H\sect N},\\
		\th_\res : \frac{N}{M\sect N} &\iso \frac{N}{H\sect N}.\\
	\end{align*} 
	\par The fact that $(L|K,\nu)$ is purely wild and $(E|K,\nu)$ is tame implies that they are valuation disjoint. Hence we have that $(L.E|E,\nu)$ is a purely wild extension [Lemma \ref{Lemma valn disj tame and purely wild}]. So $L.E\sect K^r = E$ and thus $G_r.(H\sect N)=N$.
	\pars We first suppose that $E\subsetneq (L.E|K)^\ab$, that is, $\D \subsetneq \frac{N}{N\sect H}$. Define the continuous homomorphism $\chi:G\longrightarrow \G^\ab$ as a composition of the following homomorphisms,
	\[ \begin{tikzcd}
	G \ar[r, "\pi_H"] &\frac{G}{H\sect N} = \G \ar[r, "\pi"] &\G^\ab = \frac{\G}{\D}\\
	\end{tikzcd} \]
	where $\pi_H$ and $\pi$ denote the canonical projections. Now $\pi_H(G_r) = \frac{G_r.(H\sect N)}{H\sect N} = \frac{N}{H\sect N}$. Thus $\chi(G_r) = \pi(\frac{N}{H\sect N})$. Observe that the fact that $\D \subsetneq \frac{N}{N\sect H}$ implies that $\pi(\frac{N}{N\sect H}) \neq \{1\}$. Hence $G_r \nsubseteq \ker\chi$. 
	\pars Define the continuous homomorphism $\phi:G\longrightarrow \G^\ab$ as a composition of the following homomorphisms,
	\[ \begin{tikzcd}
	G \ar[r, "\pi_M"] &\frac{G}{M\sect N} \ar[r, "\th"] &\frac{G}{H\sect N} = \G \ar[r, "\pi"] &\G^\ab = \frac{\G}{\D}\\
	\end{tikzcd} \]
	where $\pi_M$ and $\pi$ denote the canonical projections. Then $G_r \subseteq M\sect N \subseteq \ker\phi$. 
	\pars Define the map $\chi^\prime:G\longrightarrow \G^\ab$ by $\chi^\prime(g):= \chi(g)\phi(g)$ for each $g\in G$. The fact that $\G^\ab$ is an abelian group implies that $\chi^\prime$ is a continuous homomorphism. Set $N^\prime:= \ker\chi^\prime$. Now $\chi^\prime(G_r) = \chi(G_r)=\pi(\frac{N}{H\sect N})$. Again, $\chi(N)=\pi(\frac{N}{N\sect H}) = \phi(N)$ and thus $\chi^\prime(N) \subseteq \pi(\frac{N}{N\sect H})$. The fact that $G_r\subseteq N$ then implies that,
	\[ \chi^\prime(G_r)=\pi(\frac{N}{N\sect H}) = \chi^\prime(N). \]
   $\chi^\prime(G)$ is a subgroup of $\G^\ab$ and any subgroup of $\G^\ab$ is of the form $T^\prime/\D$, where $T^\prime$ is a subgroup of $\G$ containing $\D$. Again any subgroup of $\G$ is of the form $\frac{T}{N\sect H}$ where $T$ is a subgroup of $G$ containing $N\sect H$. So we can set $\chi^\prime(G) = \pi(\frac{T}{N\sect H})$ where $T$ is a subgroup of $G$ containing $N\sect H$. The fact that $\pi(\frac{N}{N\sect H}) = \chi^\prime(N)\subseteq \chi^\prime (G)$ implies that $N\subseteq T$. Again, observe that $\frac{G}{N\sect H} = \frac{H}{N\sect H} \rtimes \frac{N}{N\sect H}$ and that $\frac{N}{N\sect H} \trianglelefteq \frac{G}{N\sect H}$. Then Lemma \ref{Lemma semidirect product} implies that,
   \[ \frac{T}{N\sect H} = (\frac{H}{N\sect H} \sect \frac{T}{N\sect H}) \rtimes \frac{N}{N\sect H}. \]
	Define,
	\[ H^\prime:= (\chi^\prime)^{-1} (\pi(\frac{H}{N\sect H} \sect \frac{T}{N\sect H})). \]
	Then $H^\prime$ is a subgroup of $G$ and $N^\prime \subseteq H^\prime$. Further,
	\begin{align*}
		\chi^\prime(G_r.H^\prime)&=\chi^\prime(G_r).\chi^\prime(H^\prime)=\pi(\frac{N}{N\sect H}).\pi(\frac{H}{N\sect H} \sect \frac{T}{N\sect H})\\
		&=\pi(\frac{N}{N\sect H}. (\frac{H}{N\sect H} \sect \frac{T}{N\sect H})) = \pi(\frac{T}{N\sect H})=\chi^\prime(G).
	\end{align*}
	The fact that $G_r$ is a normal subgroup of $G$ implies that $G_r.H^\prime$ is also a subgroup of $G$. Hence from Lemma \ref{Lemma surjective homo subgroup} we obtain that $G= G_r.H^\prime. N^\prime = G_r.H^\prime$. Again,
	\begin{align*}
		&\chi^\prime(N\sect H^\prime) \subseteq \chi^\prime(N) \sect \chi^\prime(H^\prime) = \pi(\frac{N}{N\sect H}) \sect \pi(\frac{H}{N\sect H} \sect \frac{T}{N\sect H})\\
		&\Longrightarrow \chi^\prime(N\sect H^\prime) \subseteq \pi(\frac{N}{N\sect H}) \sect \pi(\frac{H}{N\sect H}) \sect \pi(\frac{T}{N\sect H}) = \pi(\frac{N}{N\sect H}) \sect \pi(\frac{H}{N\sect H}).
	\end{align*}
	Let $n\in \frac{N}{N\sect H}$ and $h\in \frac{H}{N\sect H}$ such that $\pi (n) = \pi (h)$. Then $h^{-1}n\in \ker\pi\subset \frac{N}{N\sect H}$ and thus $h\in \frac{N}{N\sect H} \sect \frac{H}{N\sect H} = \{1\}$. Hence, $\pi(\frac{N}{N\sect H}) \sect \pi(\frac{H}{N\sect H}) = \{1\}$. Consequently, 
	\[ N\sect H^\prime \subseteq \ker\chi^\prime = N^\prime. \] 
	
	\pars We have that $G_r.H=G_r.H^\prime=G$. From [Theorem 2.2 (ii), \ref{Kuh-Pank-Roquette}] we can then observe that there exist closed subgroups $A \subseteq H$ and $A^\prime\subseteq H^\prime$ such that, 
	\[ A.G_r=G=A^\prime.G_r, \, \, A\sect G_r=\{1\}=A^\prime\sect G_r. \]
	Suppose that $A$ and $A^\prime$ are conjugate subgroups of $G$. Then $A\sect N$ and $A^\prime\sect N$ are conjugate subgroups of $G$. Now, $A\sect N\subseteq H\sect N\subseteq \ker\chi$ and $A^\prime\sect N\subseteq H^\prime\sect N \subseteq \ker\chi^\prime$. Hence, 
	\[ \chi(A\sect N) = \{1\} = \chi^\prime(A\sect N) \Longrightarrow A\sect N \subseteq \ker\phi\Longrightarrow (A\sect N).(M\sect N) \subseteq \ker\phi.  \]
	Let $W:= (K^\sep)^A$. The fact that $A.G_r =G$ implies that $(W|K,\nu)$ is a purely wild extension. Hence $(E|K,\nu)$ and $(W|K,\nu)$ are valuation disjoint extensions. Thus $(W.E|E,\nu)$ is a purely wild extension and hence $W.E\sect K^r = E$. Thus we have that $G_r.(A\sect N) = N$. Then the fact that $G_r\subseteq M\sect N$ implies that,
	\[ N=G_r.(A\sect N) \subseteq (M\sect N).(A\sect N) \subseteq N \Longrightarrow N= (M\sect N).(A\sect N) \Longrightarrow N\subseteq \ker\phi.  \]
	However $\phi(N) = \pi(\frac{N}{N\sect H}) \neq\{1\}$ which thus yields a contradiction. So $A$ and $A^\prime$ are not conjugate subgroups of $G$. Since $A$ and $A^\prime$ are group complements of $G_r$ within $G$, we have that $L_A:= (K^\ac)^A$ and $L_{A^\prime}:= (K^\ac)^{A^\prime}$ are maximal purely wild extensions of $(K,\nu)$ [Theorem 4.3, \ref{Kuh-Pank-Roquette}]. Now from Galois theory it follows that $A$ and $A^\prime$ not being conjugates over $G$ imply that $W= (K^\sep)^A = K^\sep \sect L_A$ and $W^\prime = (K^\sep)^{A^\prime} = K^\sep \sect L_{A^\prime}$ are not $K$-isomorphic. Hence $L_A$ and $L_{A^\prime}$ are not isomorphic over $K$ and thus $(K,\nu)$ does not admit a unique maximal purely wild extension.
	
	\parm We now suppose that $E=(L.E|K)^\ab$. Now $(L.E|E,\nu)$ is a nontrivial purely wild Galois extension, hence a $p$-extension. Thus $L.E|E$ is a tower of Galois extensions of degree $p$ [Lemma 5.9, \ref{Kuh vln model}]. Hence $E\subsetneq (L.E|E)^\ab$. Now $(F.E|E,\nu)$ is a nontrivial Galois tame extension such that $\Gal(F.E|E) \iso \Gal(L.E|E)$. So the valued field $(E,\nu)$ satisfies the assumptions of the theorem. The fact that $E\subsetneq (L.E|E)^\ab$ then implies that $(E,\nu)$ does not admit a unique maximal purely wild extension.   
\end{proof}

Observe that $L|K$ and $E|K$ are linearly disjoint over $K$. The conditions $\Gal(L.E|K)\iso\Gal(F.E|K)$ and $(L|K,\nu)$ is a nontrivial purely wild extension then imply that $p$ divides $[F.E:K]$. Now $F.E\subset K^r$ and $\nu K^r/\nu K$ is a $p^\prime$-group. Hence it follows that the residue field $K\nu$ admits a separable extension whose degree is a multiple of $p$.  

\begin{Example}
	Let $(k,\nu)$ be the valued field $\QQ$ equipped with the $p$-adic valuation $\nu_p$ for an odd prime $p$. Let $(K,\nu)$ denote the henselization of $k$ with respect to $\nu_p$. Then $\nu K=\ZZ$ and $K\nu = \FF_p$. Set $E:= K(\zeta)$ where $\zeta$ is a primitive $p$-th root of unity. Then $[E:K]\leq [\QQ(\zeta):\QQ] = p-1$. From Lemma \ref{lemma C^{p-1} = p} we have that there exists $\th\in E$ such that $\nu\th = \frac{1}{p-1}$. Thus $\frac{1}{p-1}\ZZ \subseteq \nu E$. It then follows from the Lemma of Ostrowski that $(E|K,\nu)$ is a tame extension with $[E:K] = [\nu E:\nu K] = p-1$. Further observe that $E|K$ is an abelian Galois extension where $\Gal(E|K) = C_{p-1}$ is the cyclic group of order $p-1$.
\pars Let $a>p^2$ be such that $a\equiv 1 \bmod p$. Set $F:= K(\a)$ where $\a^p = a$. Then $F.E = E(\a)$ is a Kummer extension with $\nu\a=0$. From Lemma \ref{Lemma value group bound Kummer extn} it then follows that $\nu(\a-E) \leq \frac{1}{p-1}$. Now $a-1 = np^2$ where $n\geq 1$. Hence $\nu(a-1) \geq 2 > \frac{p}{p-1}$. We have an expression,
\[ (\a-1)^p = (a-1) + p(\a-1)g \]
where $\nu g \geq 0$. Then the fact that $\nu(a-1) > \frac{p}{p-1} \geq p\nu(\a-1)$ implies that $p\nu (\a-1) = \nu [p(\a-1)g]$. Consequently, $\nu(\a-1) = \frac{1}{p-1}$. From Theorem \ref{Theorem classification kummer} we then obtain that $(F.E|E,\nu)$ is a defectless unramified extension. In particular, $F\subseteq F.E \subseteq E^r = K^r$ and hence $(F|K,\nu)$ is a nontrivial tame extension. 
\pars Let $L:= K(\b)$ where $\b^p=p$. Thus $\nu\b = \frac{1}{p}$ and thus $(L|K,\nu)$ is a nontrivial purely wild extension. 
\pars Observe that $F.E$ is the splitting field of $X^p-a$ over $K$ and hence $F.E|K$ is a Galois extension. We have that $\Gal(F.E|K) = C_p \ltimes C_{p-1}$ and $\Gal(F.E|E) = C_p$ where $C_p$ is the cyclic group of order $p$. Similarly, $\Gal(L.E|K) = C_p\ltimes C_{p-1}$ and $\Gal(L.E|E) = C_p$. The fact that $E|K$ is abelian Galois implies that $E\subseteq (L.E|K)^\ab$. So, $[(L.E|K)^\ab:K] = c(p-1)$ where $c$ divides $p$. If $c=p$ then $L.E|K$ is an abelian extension, which yields a contradiction. So $c=1$ and consequently $E= (L.E|K)^\ab$. Observe that $E$, $F$ and $L$ satisfy the assumptions of Theorem \ref{Theorem non-unique}. It then follows that $(E,\nu)$ does not admit a unique maximal purely wild extension.     
\end{Example}

The conclusions of the above example also follow from Corollary \ref{Corollary Galois nonunique}.

\begin{Theorem}
	Let $(L|K,\nu)$ be a nontrivial separable purely wild extension of henselian valued fields with $\ch K\nu=p>0$. Set $H:= \Gal(K^\sep|L)$ and $F:= (K^\sep)^{[G_r,H]}$. Suppose that $L$ and $F$ are not linearly disjoint over $K$. Then $(K,\nu)$ does not admit a unique maximal purely wild extension if the residue field $K\nu$ is not Artin-Schreier closed. 
\end{Theorem}

\begin{proof}
	Let $N:= \Gal(K^\sep|F)$. We know from Galois theory that $N$ is the closure of $[G_r,H]$ in $G$ and hence $[G_r,H] \subseteq N$. The fact that $G_r.H=G$ implies that $[G_r,H] \trianglelefteq G$. Hence $N$ is a closed normal subgroup of $G$ [Proposition 1.5, Chapter 6, \ref{Karpilovsky}] and consequently $F|K$ is a Galois extension. Then $L$ and $F$ not being linearly disjoint over $K$ implies that $L\sect F\neq K$. Thus, 
	\[ K=L\sect K^r \subsetneq L\sect F \subseteq L\Longrightarrow (L\sect F|K,\nu) \text{ is a nontrivial purely wild extension. } \]
	So $K^r|K$ and $(L\sect F)|K$ are valuation disjoint. Hence $[(L\sect F).K^r:K^r] = [(L\sect F):K] >1$. Further, $G_r$ being a normal subgroup of $G$ implies that $[G_r,H] \subseteq G_r$. Consequently,  $N\subseteq G_r$ and hence $K^r \subseteq F$. Now, 
	\[ K^r \subsetneq (L\sect F).K^r \subseteq L.K^r \sect F.K^r = L^r \sect F. \]
	Thus we have that $(G_r\sect H).N \subsetneq G_r$ and hence $(G_r\sect H).N$ is a proper closed subgroup of $G_r$. Hence there exists a maximal closed subgroup $W$ of $G_r$ such that $(G_r\sect H).[G_r,H] \subseteq (G_r\sect H).N \subseteq W \subsetneq G_r$. The fact that $G_r$ is a pro-$p$-group now implies that $W$ is a normal subgroup of $G_r$ with $[G_r:W]=p$. 
	\pars Consider the continuous homomorphism $\chi:G_r \longrightarrow G_r/W \iso \ZZ/p\ZZ $ with $\ker\chi =W$. We extend the map to $G= G_r.H$ by defining, 
	\begin{align*}
	\chi:G &\longrightarrow \ZZ/p\ZZ \\
	gh&\longmapsto \chi(g)\\
	\end{align*}
	where $g\in G_r$ and $h\in H$. Observe that the fact $[G_r,H] \subseteq W$ implies that $hgh^{-1}g^{-1} \in W$ and hence $\chi(hgh^{-1}) = \chi(g)$ for each $g\in G_r$ and $h\in H$. We now check that the map is well-defined. Indeed for $g_i\in G_r$ and $h_i\in H$ we have that,
	\[ g_1h_1=g_2h_2 \Longrightarrow g_2^{-1}g_1=h_2h_1^{-1} \in G_r\sect H \Longrightarrow \chi(g_1) =\chi(g_2) \Longrightarrow \chi(g_1h_1)=\chi(g_2h_2). \]
	Again using the expression $g_1h_1g_2h_2=g_1h_1g_2h_1^{-1}h_1h_2$ and observing that $\chi$ is a group homomorphism on $G_r$ we obtain that, 
	\[ \chi(g_1h_1g_2h_2)=\chi(g_1h_1g_2h_1^{-1}) =\chi(g_1)+\chi(h_1g_2h_1^{-1}) =\chi(g_1)+\chi(g_2)=\chi(g_1h_1)+\chi(g_2h_2). \]
	Thus $\chi:G\longrightarrow \ZZ/p\ZZ$ is a surjective group homomorphism with $H \subseteq \ker\chi$. Hence $G_r \nsubseteq \ker\chi$. The fact that $[G:\ker\chi]=p$ now implies that $G_r.\ker\chi = G$.
	
	\pars  The residue field $K\nu$ not being Artin-Schreier closed admits an Artin-Schreier extension $K\nu(\th)|K\nu$. Set $G\nu := \Gal(K\nu^\sep|K\nu)$ and set $T\nu:= Gal(K\nu^\sep |K\nu(\th))$. Then $T\nu$ is a normal subgroup of $G\nu$ with $[G\nu:T\nu]=p$. We have that $G\nu \iso G/G_i$. Thus normal closed subgroups of $G\nu$ correspond to normal closed subgroups of $G$ containing $G_i$. Let $T\nu$ correspond to $T/G_i$ where $T$ is a normal subgroup of $G$ with $[G:T]=[G\nu:T\nu]=p$. We have the chain,
	\[ G_r \trianglelefteq G_i \trianglelefteq T \trianglelefteq G. \]
	Consider the continuous surjective homomorphism,
	\[ \phi:G\longrightarrow G/T \iso \ZZ/p\ZZ . \]
	Thus $G_r \subseteq T= \ker \phi$ and hence $H\nsubseteq \ker\phi$.
	
	\pars Define the continuous homomorphism $\psi:G\longrightarrow \ZZ/p\ZZ$ given by $\psi = \chi + \phi$. Then $G_r \nsubseteq \ker\psi$, hence $\psi$ is surjective and $G_r.\ker \psi = G$. From [Theorem 2.2 (ii), \ref{Kuh-Pank-Roquette}] we have that there exist closed subgroups $A \subseteq \ker\chi$ and $B\subseteq\ker\psi$ such that, 
	\[ A.G_r=G=B.G_r, \, \, A\sect G_r=\{1\}=B\sect G_r. \]
	If $A$ and $B$ are conjugate subgroups over $G$ then we have that $\chi(A)=\psi(A)=0$ and thus $\phi(A)=0$. Again we have that $\phi(G_r)=0$. Then the fact that $A.G_r=G$ implies that $\phi(G)=0$ which is a contradiction. So $A$ and $B$ are not conjugates in $G$. With the same arguments as in the proof of Theorem \ref{Theorem non-unique} it now follows that $(K,\nu)$ does not admit a unique maximal purely wild extension.	  
\end{proof}

\pars The characters $\chi, \, \phi$ and $\psi$ are \enquote{admissible} in the language of [\ref{Kuh-Pank-Roquette}]. Following the same arguments as in the proof of Proposition 8.5 of [\ref{Kuh-Pank-Roquette}] we arrive at the following result: 
\begin{Theorem}
	Let $(L|K,\nu)$ be a nontrivial separable immediate extension of henselian valued fields with $\ch K\nu=p>0$. Set $H:= \Gal(K^\sep|L)$ and $F:= (K^\sep)^{[G_r,H]}$. Suppose that $L$ and $F$ are not linearly disjoint over $K$. Then $(K,\nu)$ does not admit a unique maximal immediate extension if the residue field $K\nu$ is not Artin-Schreier closed.
\end{Theorem}

If $L|K$ is Galois, then $H$ is a closed normal subgroup of $G$ and thus $[G_r,H] \subseteq G_r\sect H$. The following corollary is immediate:

\begin{Corollary}\label{Corollary Galois nonunique}
	Let $(L|K,\nu)$ be a nontrivial Galois purely wild extension of henselian valued fields with $\ch K\nu=p>0$. Suppose that the residue field $K\nu$ is not Artin-Schreier closed. Then $(K,\nu)$ does not admit a unique maximal purely wild extension. If $(L|K,\nu)$ is immediate, then $(K,\nu)$ does not admit a unique maximal immediate algebraic extension.
\end{Corollary}


\end{document}